\documentclass[11pt]{article}

\usepackage{amstext,amsfonts,amsthm,graphicx,amssymb,amscd,epsfig}
\usepackage{amsmath}
\usepackage[ansinew]{inputenc}    %

\newtheorem{definition}{Definition}[section]
\newtheorem{rem}[definition]{Remark}
\newtheorem{prop}[definition]{Proposition}
\newtheorem{lem}[definition]{Lemma}
\newtheorem{coro}[definition]{Corollary}
\newtheorem{teo}[definition]{Theorem}
\newtheorem{ex}[definition]{Example}
\newcommand{\C}{\;\mbox{{\sf I}}\!\!\!C}

\newcommand{\R}{I\!\!R}

\newcommand{\K}{I\!\!K}

\newfont{\bbb}{msbm10 scaled\magstephalf}     
\def\C{\mbox{\bbb C}}
\def\K{\mbox{\bbb K}}

\def\cod{\operatorname{cod}}

\def\R{\mbox{\bbb R}}

\title{\bf Classifying codimension two multigerms}

\author{R. Oset Sinha,  M. A. S. Ruas and R. Wik Atique}

\date{}

\begin{document}

\maketitle
\begin{abstract}
 We generalise the operations of augmentation and concatenations defined in \cite{robertamond} in order to obtain
multigerms of analytic (or smooth) maps $(\mathbb
K^n,S)\rightarrow(\mathbb K^p,0)$ with $\mathbb K=\mathbb C$ or
$\mathbb R$ from monogerms and some special multigerms. We then
prove that any corank 1 codimension 2 multigerm in Mather's nice
dimensions $(n,p)$ with $n\geq p-1$ can be constructed using
augmentations and these operations.
\end{abstract}

\renewcommand{\thefootnote}{\fnsymbol{footnote}}
 \footnote[0]{The
first author is partially supported by FAPESP grant no. 2010/01501-5
and DGCYT and FEDER grant no. MTM2009-08933. The second and third
authors are partially supported by FAPESP grant no. 2008/54222-6.
The second author is partially supported by CNPq grant no.
303774/2008-8.}

\section{Introduction}

The classification of singularities of map germs from $\K^n$ to
$\K^p$ has been one of the main areas of research in Singularity
Theory for the last decades. The foundations are the fundamental
works of Whitney, Mather and Thom on classification of stable maps
followed by Arnold's classification of simple singularities of
functions in \cite{Arnold3}. Since then, complete classifications up
to certain codimension for certain pairs $(n,p)$ have been carried
out by many authors (\cite{rieger}, \cite{riegerruas}, \cite{mond},
\cite{Goryunov}, \cite{marartari}, \cite{klotz}, \cite{roberta},
...) and it is still an active field of research.

The bibliography related to the classification of multigerms is less
abundant. The first reference is Mather's classification of stable
multigerms \cite{mather}. A list of multigerms from $\R^2$ to $\R^3$
including codimension 1 singularities is given by Goryunov
(\cite{Goryunov3}, \cite{Goryunov2}). Hobbs and Kirk
(\cite{hobbskirk}) give a thorough classification for this case. The
third author obtains in \cite{roberta} the list of simple multigerms
from $\C^2$ to $\C^3$ and gives a method that can be applied to the
case $\C^n$ to $\C^{n+1}$. Normal forms for multigerms from the
plane to the plane are given by several authors. A good account of
this is \cite{Ohmoto}. More recently, in \cite{mio}, the first
author and Romero Fuster give normal forms of multigerms up to
codimension 2 from $\R^3$ to $\R^3$ using a geometric method based
on the theory of contact between submanifolds developed by Montaldi
in \cite{Mon}. This method is used to give candidates for the
different $\mathcal A$-classes and then normal forms are given,
being able to prove that they have the announced codimension.

The classification of multigerms using the classical Singularity
Theory techniques such as the complete transversal's method can be
hard to deal with. In \cite{roberta} the $\mathcal A$-classification
of multigerms from $\C^n$ to $\C^{n+1}$ is reduced to a much simpler
$\mathcal K$-classification. However, in other dimensions, it is
still a very hard task. Therefore, operations to obtain germs in
certain dimensions from germs with fewer branches in lower
dimensions have been developed. The concepts of augmentations and
monic and binary concatenations appear in \cite{robertamond}, where
Cooper, Mond and Wik Atique show that any codimension 1 multigerm
can be obtained using these operations. Further developments on
these operations are given by Houston (\cite{houston2},
\cite{houston}).

As can be seen from the classification in \cite{mio}, these
operations do not give the complete list of codimension 2
multigerms.

In this paper we generalise these operations in order to obtain a
complete list of multigerms of codimension 2. The first operation
introduced merges two of the earlier ones, it is a simultaneous
augmentation and monic concatenation. Then, a generalised definition
of concatenation is given which includes both the monic and binary
concatenations as particular cases. Two other examples of this
family of operations are studied, namely cuspidal concatenation and
double fold concatenation. We show that for $(n,p)$ in Mather's nice
dimensions and $n\geq p-1$ all the multigerms of corank 1 and
codimension 2 are obtained using these operations starting from
monogerms and some special multigerms. We are therefore giving a
method to obtain new classifications of multigerms for any $(n,p)$
in Mather's nice dimensions, including the complete lists of
singularities up to codimension 2.

The study of multigerms is not only important for classification
purposes but is also necessary for other research lines such as the
study of topological invariants. For instance, to obtain the first
order local invariants of stable maps it is essential to know the
multigerms up to codimension 2 (see for example \cite{Goryunov2},
\cite{Ohmoto}, \cite{yamamoto}, \cite{mio} or \cite{catiana}). Also,
our methods provide a source of new examples to test the Mond
conjecture which states that the codimension of a germ is less than
or equal to its image Milnor number.

The paper is organised as follows: In \S2 we fix our notation and
give some basic results and definitions. In \S3 we define
augmentations, state some known results and give a characterisation
to identify augmented multigerms having only stable branches. In \S4
we define the new operations, give formulae to calculate the
codimension of the resulting multigerms and several examples. In \S5
we prove that any codimension 2 multigerm in Mather's nice
dimensions with $n\geq p-1$ is obtained using these operations.
Finally, in \S6 we give an example of how to obtain all the
codimension 2 multigerms using these operations, namely the
classification of codimension 2 multigerms from $\C^3$ to $\C^3$
obtained in \cite{mio}.

\emph{Acknowledgements:} The authors would like to thank the referee
for a careful reading of the manuscript and many valuable
suggestions.

\section{Notation}

Let $\mathcal O_{n}^p$ be the vector space of map germs with $n$
variables and $p$ components. When $p=1$, $\mathcal O_{n}^1=\mathcal
O_n$ is the local ring of germs of functions in $n$-variables and
$\mathcal M_n$ its maximal ideal. The set $\mathcal O_{n}^p$ is a
free $\mathcal O_n$-module of rank $p$. A multigerm is a germ of an
analytic (complex case) or smooth (real case) map
$f=\{f_1,\ldots,f_r\}:(\K^n,S)\rightarrow (\K^p,0)$ where
$S=\{x_1,\ldots,x_r\}\subset \K^n$, $f_i:(\K^n,x_i)\rightarrow
(\K^p,0)$ and $\K=\C$ or $\R$. Let $\mathcal M_n\mathcal O_{n,S}^p$
be the vector space of such map germs. Let $\theta_{\mathbb K^n,S}$
and $\theta_{\mathbb K^p,0}$ be the $\mathcal O_n$-module of germs
at $S$ of vector fields on $\K^n$ and $\mathcal O_p$-module of germs
at 0 of vector fields on $\K^p$ respectively. We denote them by
$\theta_n$ and $\theta_p$. Let $\theta(f)$ be the $\mathcal
O_n$-module of germs $\xi:(\K^n,S)\rightarrow T\K^p$ such that
$\pi_p\circ \xi=f$ where $\pi_p:T\K^p\rightarrow \K^p$ denotes the
tangent bundle over $\K^p$.

Define $tf:\theta_n\rightarrow \theta(f)$ by $tf(\chi)=df\circ\chi$
and $wf:\theta_p\rightarrow \theta(f)$ by $wf(\eta)=\eta\circ f$.
The $\mathcal A_e$-tangent space of $f$ is defined as $T\mathcal A_e
f=tf(\theta_n)+wf(\theta_p)$. Finally we define the $\mathcal
A_e$-codimension of a germ $f$, denoted by $\mathcal A_e$-cod($f$),
as the $\K$-vector space dimension of
$$NA_e(f)=\frac{\theta(f)}{T\mathcal A_e f}.$$ We refer to Wall's
survey article \cite{wall} for general background on the theory of
singularities.

Following Damon in \cite{damon}, a transverse fibre square is a
diagram

\begin{equation*}
\begin{CD}
X @>{F}>> Y\\ @AA{j}A        @A{i}AA   \\
X_0 @>{f}>> Y_0
\end{CD}
\end{equation*}
in which $i\pitchfork F$, $X_0\simeq X\times_{Y} Y_0$ and $f$ is
right equivalent to the projection $X\times_{Y} Y_0\rightarrow Y_0$.
The map-germ $f$ is called the pull-back of $F$ by $i$ and is
denoted by $i^*(F)$. Any germ $f$ is a pull-back of a stable
$q$-parameter unfolding $F$ by the natural inclusion
$i:(\K^p,0)\rightarrow (\K^{p}\times\K^q,0)$. Damon proved that
$\mathcal A_e$-cod$(f)=\mathcal K_{D(F),e}$-cod$(i)$, where
$$\mathcal K_{D(F),e}-\mbox{cod}(i)=\dim_{\mathbb K} N\mathcal
K_{D(F),e}(i)=\dim_{\mathbb K}
\frac{\theta(i)}{ti(\theta_p)+i^*(Derlog(D(F)))}$$ where $D(F)$ is
the discriminant of $F$ and $Derlog(V)$ represents the $\mathcal
O_p$-module of tangent vector fields to $V$.

\begin{definition}

i) A vector field germ $\eta\in \theta_p$ is called \textit{liftable
over $f$}, if there exists $\xi\in\theta_n$ such that
$df\circ\xi=\eta\circ f$ ($tf(\xi)=wf(\eta)$). The set of vector
field germs liftable over $f$ is denoted by Lift($f$) and is an
$\mathcal{O}_p$-module.

ii) Let $\widetilde{\tau}(f)=ev_0$(Lift($f$)) be the evaluation at
the origin of elements of $Lift(f)$.
\end{definition}

In general $Lift(f)\subseteq Derlog(V)$ when $V$ is the discriminant
of an analytic $f$. We have an equality when $\K=\C$ and $f$ is
complex analytic.

The set $\widetilde{\tau}(f)$ is the tangent space to the well
defined manifold containing 0 along which the map $f$ is trivial
(i.e. the analytic stratum). Notice that if $f$ is stable then
$\widetilde{\tau}(f)=\tau(f)$ in Mather's sense (namely
$\tau(f)=ev_0[wf^{-1}\{f^*\mathcal
M_p\theta(f)+tf(\theta_{n,S})\}]$). See \cite{houston} or
\cite{robertamond} for some basic properties of $\widetilde\tau(f)$.

From here on we will consider only minimal corank germs.

\section{Augmentations of smooth mappings}

\begin{definition}
Let $h:(\mathbb{K}^n,S)\rightarrow (\mathbb K^p,0)$ be a map-germ
with a 1-parameter unfolding $H:(\mathbb K^n\times\mathbb
K,S\times\{0\})\rightarrow (\mathbb K^p\times\mathbb K,0)$ which is
stable as a map-germ, where $H(x,\lambda)=(h_{\lambda}(x),\lambda)$,
such that $h_0=h$. Let $g:(\mathbb K^q,0)\rightarrow (\mathbb K,0)$
be a function-germ. Then, the \textit{augmentation of h by H and g}
is the map $A_{H,g}(h)$ given by $(x,z)\mapsto (h_{g(z)}(x),z)$.
\end{definition}

\begin{teo}(\cite{houston2},\cite{houston})
 $$\mathcal{A}_e-cod(A_{H,g}(h))\geq\mathcal{A}_e-cod(h)\tau(g),$$ where $\tau$ is the Tjurina number and equality is reached when $g$ is quasihomogeneous.
\end{teo}

\begin{teo}(\cite{houston})\label{aughous}
Suppose that $f:(\mathbb K^n,S)\rightarrow (\mathbb K^p,0)$ is
non-stable and has a 1-parameter stable unfolding $F$. Then
$$q=\dim_{\mathbb K}\widetilde{\tau}(F)\geq 1\Leftrightarrow f
\text{ is an augmentation}.$$ More precisely, on the right hand
side, $f\sim_{\mathcal{A}}A_{H,g}(h)$ for some  $h:(\mathbb
K^{n-q},S')\rightarrow (\mathbb K^{p-q},0)$, a smooth map-germ with
a 1-parameter stable unfolding $H$, and $g:(\mathbb
K^q,0)\rightarrow (\mathbb K,0)$ a function, $q\geq 1$.
\end{teo}

A germ that is not an augmentation is called primitive. When all the
branches of a multigerm are stable we give a more geometric
characterisation for augmented germs.

\begin{definition}
Given a multigerm $f=\{f_1,\ldots,f_r\}$ we say that its branches
are totally non-transverse if for every $i=1,\ldots,r$ there exists
$\eta_i\in Lift(f_i)$, $\eta_i(0)\neq 0$, such that
$\eta_i(0)=\eta_j(0)$ for all $j\neq i$.
\end{definition}

From the definition it can be seen that the fact of the branches
being totally non-transverse is equivalent to
$\dim(\widetilde\tau(f_1)\cap\ldots\cap\widetilde\tau(f_r))=\dim(\tau(f_1)\cap\ldots\cap\tau(f_r))\geq
1$.

\begin{prop}\label{totnon}
Let $f=\{f_1,\ldots,f_r\}:(\mathbb K^n,S)\rightarrow (\mathbb
K^p,0)$ be a non-stable multigerm, $|S|>1$, with $f_i$ stable for
all $i=1,\ldots,r$. Suppose that $f$ admits a 1-parameter stable
unfolding $F$. If $f$ is an augmentation then its branches are
totally non-transverse.
\end{prop}
\begin{proof}
As in Theorem 3.3, one can take a primitive $h$ such that
$f\sim_{\mathcal{A}}A_{H,g}(h)$. Since $f$ is non-stable $g$ is not
a submersion.

Following the proof of Theorem 4.6 in \cite{houston},
$f(x,z)=(h_{g(z)}(x),z)$, with $x\in\mathbb K^{n-q}$ and
$z\in\mathbb K^{q}$, has the 1-parameter stable unfolding
$F(x,z,\lambda)=(h_{g(z)+\lambda}(x),z,\lambda)=(X,Z,\Lambda)$. Also
$F\sim_{\mathcal{A}}H\times Id_q$, that is, there exist germs of
diffeomorphisms $\psi$ and $\phi$ such that
$F\circ\phi=\psi\circ(H\times Id_q)$, namely,
$\psi(X,Z,\Lambda)=(X,\Lambda,Z+g(\Lambda))$ and
$\phi(x,z,\lambda)=(x,\lambda,z-g(\lambda))$. Since $g$ is not a
submersion it follows that $d\psi_0$ is a permutation of the
identity. Using the basic properties of $\widetilde\tau$ we have
that $$\widetilde{\tau}(F)=d\psi_0(\widetilde{\tau}(H\times
Id_q))=d\psi_0(\widetilde{\tau}(H)\times T_0\mathbb
K^q)=d\psi_0(\{0\}\times T_0\mathbb K^q)=\{0\}\times T_0\mathbb
K^q\times\{0\}$$ (and therefore $\dim_{\mathbb
K}\widetilde{\tau}(F)=q\geq 1$).

Since $\dim_{\mathbb K}\widetilde{\tau}(F)\geq 1$, there exists
$\eta\in Lift(F)$ such that $\eta(0)\neq 0$. Clearly $\eta\in
Lift(F_i)$ for each branch $F_i$ of $F$. There exists $k$,
$p-q+1\leq k\leq p$, such that the $k$-th component of $\eta$ is non
zero, that is, $\eta_k(0)\neq 0$. Let $\xi^i$ be such that
$dF_i(\xi^i(x,z,\lambda))=\eta(F_i(x,z,\lambda))$, $i=1,\ldots,r$.
Define $\overline\xi^i(x,z)=(\xi^i_1(x,z,0),\ldots,\xi^i_n(x,z,0))$
and $\overline\eta(X,Z)=(\eta_1(X,Z,0),\ldots,\eta_p(X,Z,0))$. By
evaluating this system of equations in $\lambda=0$, from the first
$p$ equations we get
$$df_i(\overline\xi^i(x,z))+\partial_{\lambda}((h_{g(z)+\lambda}(x),z))|_{\lambda=0}\,\xi^i_{n+1}(x,z,0)=\overline\eta(f_i(x,z))$$
where $$\xi^i_{j+(n-p)}(x,z,0)=\eta_j(f_i(x,z),0),\,\,\, p-q+1\leq
j\leq p$$ and from the last equation we get
$$\xi^i_{n+1}(x,z,0)=\eta_{p+1}(f_i(x,z),0).$$ Since $f_i$ is
stable, there exist vector fields $v^i=(v^i_1,\ldots,v^i_n)$ and
$\gamma^i=(\gamma^i_1,\ldots,\gamma^i_p)$ such that
$$\partial_{\lambda}((h_{g(z)+\lambda}(x),z))|_{\lambda=0}\,\xi^i_{n+1}(x,z,0)=df_i(v^i)+\gamma^i(f_i).$$
Since $\xi^i_{n+1}(0)=\eta_{p+1}(0)=0$,
$\partial_{\lambda}((h_{g(z)+\lambda}(x),z))|_{\lambda=0}\,\xi^i_{n+1}(x,z,0)\in
f_i^*\mathcal M_p\theta(f_i)=f_i^*\mathcal
M_p(tf_i(\theta_n)+wf_i(\theta_p))\subset
tf_i(\theta_n)+wf_i(\mathcal M_p\theta_p)$, so we can choose
$\gamma^i(0)=0$. The system can now be written as
$df_i(\overline\xi^i+v^i)+\gamma^i(f_i)=\overline\eta(f_i)$.
Therefore $\overline\eta(X,Z)-\gamma^i(X,Z)\in Lift(f_i)$ and
$\overline\eta(0)-\gamma^i(0)=(\eta_1(0),\ldots,\eta_p(0))\neq 0$
for all $i=1,\ldots,r$.
\end{proof}

The proposition is not true if one of the branches is not stable,
since the analytic stratum of a non stable branch is $\{0\}$.

A finite set $E_1,\ldots,E_r$ of vector subspaces of a
finite-dimensional vector space $F$ has almost regular intersection
of order $k$ (with respect to $F$) if $$\cod(E_1\cap\ldots\cap
E_r)=\cod E_1+\ldots+\cod E_r-k,$$ where $\cod$ represents the
codimension. When $k=0$ we say regular intersection and when $k=1$
we say almost regular intersection. Mather characterised in
\cite{mather} stable multigerms as those where every branch is
stable and the analytic strata have regular intersection. Elementary
algebra proves the following

\begin{lem}
$E_1,\ldots,E_r$ have almost regular intersection of order $k$ if
and only if the cokernel of the natural mapping $$F\rightarrow
(F/E_1)\oplus\ldots\oplus(F/E_r)$$ has dimension $k$.
\end{lem}

\begin{prop}\label{almreg}
Let $f=\{f_1,\ldots,f_r\}:(\mathbb K^n,S)\rightarrow (\mathbb
K^p,0)$ be a non-stable multigerm, $|S|>1$, with $f_i$ stable for
all $i=1,\ldots,r$. If $f$ admits a 1-parameter stable unfolding,
then the $\widetilde\tau(f_i)$ have almost regular intersection.
Moreover, if $f$ admits a decomposition $f=\{f^1,f^2\}$ with
$f^1,f^2$ stable germs, then the converse is also true.
\end{prop}
\begin{proof}
A 1-parameter stable unfolding $F$ of $f$ restricts to 1-parameter
stable unfoldings $F_i$ of $f_i$. Since $f_i$ is stable for all
$i=1,\ldots,r$, $F_i$ is a prism on $f_i$, that is
$F_i\sim_{\mathcal A} f_i\times id$, therefore
$$T_0\mathbb K^p/\widetilde\tau(f_i)\cong T_0(\mathbb
K^p\times\mathbb K)/\widetilde\tau(F_i) \,\,\,\text{for all}\,\,\,
i=1,\ldots,r.$$

We have the following commutative diagram

\begin{equation*}
\begin{CD}
T_0\mathbb K^p @>>> \frac{T_0\mathbb K^p}{\widetilde\tau(f_1)}\oplus\dots\oplus\frac{T_0\mathbb K^p}{\widetilde\tau(f_r)}\\ @VVV        @VVV   \\
T_0(\mathbb K^p\times\mathbb K) @>>> \frac{T_0(\mathbb
K^p\times\mathbb
K)}{\widetilde\tau(F_1)}\oplus\ldots\oplus\frac{T_0(\mathbb
K^p\times\mathbb K)}{\widetilde\tau(F_r)}
\end{CD}
\end{equation*}

The right hand map is bijective and the bottom map is surjective
since $F$ is stable (and therefore the $\widetilde\tau(F_i)$ are
transversal). So the top map has cokernel of dimension 1. If it were
of dimension 0, it would be surjective and the $\widetilde\tau(f_i)$
would be transverse, which means that $f$ would be stable and this
is a contradiction.

Conversely, if they have almost regular intersection and
$f=\{f^1,f^2\}$ with $f^1,f^2$ stable then there exists a direction,
which we call $v$ such that $T_0\mathbb
K^p=(\widetilde\tau(f^1)+\widetilde\tau(f^2))\oplus\mathbb K\{v\}$.
If $f_2$ is a monogerm, then $\{F^1,F^2\}$ with $F^1=f^1\times id_t$
and $F^2=(f^2+tv,t)$ is a 1-parameter stable unfolding of $f$, since
the $\widetilde\tau(F^i)$ have regular intersection. If $f_2$ is a
multigerm, we deform any one of its branches in the way above and
the result follows too.
\end{proof}

\begin{coro}\label{aug}
Let $f=\{f_1,\ldots,f_r\}:(\mathbb K^n,S)\rightarrow (\mathbb
K^p,0)$ be a non-stable multigerm, $|S|>1$, with $f_i$ stable for
all $i=1,\ldots,r$. Suppose that $f$ admits a 1-parameter stable
unfolding $F$. Then, $$\sum_{i=1}^r \cod(\widetilde\tau(f_i))\leq
p\Leftrightarrow f \text{ is an augmentation}.$$
\end{coro}
\begin{proof}
Since $f$ has a 1-parameter stable unfolding, by Proposition
\ref{almreg} the analytic strata have almost regular intersection.

By Proposition \ref{totnon} we have that if $f$ is an augmentation
then its branches are totally non-transverse, that is, $1\leq
\dim(\widetilde\tau(f_1)\cap \ldots \cap \widetilde\tau(f_r))$. We
have
$$\sum_{i=1}^r \cod(\widetilde\tau(f_i))=\cod(\bigcap_{i=1}^r \widetilde\tau(f_i))+1=p-\dim(\bigcap_{i=1}^r \widetilde\tau(f_i))+1\leq p.$$

On the other hand, the fact of the analytic strata having almost
regular intersection, together with $\sum_{i=1}^r
\cod(\widetilde\tau(f_i))\leq p$ implies that
$\dim(\widetilde\tau(f_1)\cap \ldots \cap \widetilde\tau(f_r))\geq
1$. If we consider a 1-parameter stable unfolding
$F=\{F_1,\ldots,F_r\}$ of $f$, since $f_i$ is stable for all
$i=1,\ldots,r$, $F_i$ is equivalent to a prism on $f_i$, therefore
$\cod\widetilde\tau(F_i)=\cod\widetilde\tau(f_i)$ for all $i$. It
follows that $$\cod(\bigcap_{i=1}^r
\widetilde\tau(F_i))=\sum_{i=1}^r
\cod(\widetilde\tau(F_i))=\sum_{i=1}^r
\cod(\widetilde\tau(f_i))=\cod(\bigcap_{i=1}^r
\widetilde\tau(f_i))+1$$ and therefore
$\dim(\widetilde\tau(F))=\dim(\bigcap_{i=1}^r
\widetilde\tau(F_i))=\dim(\bigcap_{i=1}^r \widetilde\tau(f_i))\geq
1$ and by Theorem \ref{aughous} we have that $f$ is an augmentation.
\end{proof}

This result can also be obtained as a corollary of Theorem
\ref{aughous}. However, this reformulation will be used in Section
5.

\begin{ex}
\begin{enumerate}
\item[i)]Any multigerm involving two cuspidal edges in $\mathbb K^3$ can never be an augmentation because the sum of the codimension of the analytic strata is 4.

\item[ii)] Any multigerm from $\mathbb K^2$ to $\mathbb K^3$ involving a cross-cap and other immersive branches can never be an augmentation because the codimension of the analytic stratum of the cross-cap is already 3.

\item[iii)] Using the classical Arnol'd notation for multigerms in the case $n=p$ where $A_i^kA_j$ represents a multigerm with $k$ branches of type $A_i$ and a branch of type $A_j$ where all of the branches are pairwise transversal, then it is an augmentation if and only if $(ik+j)\leq p$.
\end{enumerate}
\end{ex}

\section{New operations}

\begin{definition}
Suppose $f:(\mathbb K^n,S)\rightarrow (\mathbb K^p,0)$ is non-stable
of finite $\mathcal{A}_e$-codimension and has a 1-parameter stable
unfolding $F(x,\lambda)=(f_{\lambda}(x),\lambda)$.  Let $k\geq 0$
and $g:(\mathbb K^p\times\mathbb K^k,0)\rightarrow (\mathbb
K^p\times\mathbb K,0)$ be the fold map
$(X,v)\mapsto(X,\Sigma_{j=1}^k v_j^2)$ (when $k=0$ $g(X)=(X,0)$).
Then the multigerm $\{F,g\}$ is called a monic concatenation of $f$.
\end{definition}

\begin{teo}\label{codconc}(\cite{robertamond}) The following relation holds:
$$\mathcal{A}_e-cod(\{F,g\})=\mathcal{A}_e-cod(f).$$
\end{teo}

In the previous definition we set $p+k=n+1$. We now introduce a new
operation which merges the two earlier ones, a simultaneous
augmentation and monic concatenation.

\begin{teo}\label{augconc}
Suppose $f:(\mathbb K^n,S)\rightarrow (\mathbb K^p,0)$ has a
1-parameter stable unfolding
$F(x,\lambda)=(f_{\lambda}(x),\lambda)$.  Let $g:(\mathbb
K^p\times\mathbb K^{n-p+1},0)\rightarrow (\mathbb K^p\times\mathbb
K,0)$ be the fold map $(X,v)\mapsto(X,\Sigma_{j=p+1}^{n+1} v_j^2)$.
Then,

i) the multigerm $\{A_{F,\phi}(f),g\}$, where
$\phi:\K\rightarrow\K$, has
$$\mathcal{A}_e-cod(\{A_{F,\phi}(f),g\})\geq\mathcal{A}_e-cod(f)(\tau(\phi)+1),$$
where $\tau$ is the Tjurina number of $\phi$. Equality is reached
when $\phi$ is quasi-homogeneous and $\langle
dZ(i^*(Lift(A_{F,\phi}(f))))\rangle=\langle dZ(i^*(Lift(F)))\rangle$
where $i:\mathbb K^p\rightarrow \mathbb K^{p+1}$ is the canonical
immersion $i(X_1,\ldots,X_{p})=(X_1,\ldots,X_{p},0)$ and $Z$ is the
last component of the target.

ii) $\{A_{F,\phi}(f),g\}$ has a 1-parameter stable unfolding.
\end{teo}
\begin{proof}
We write $Af$ for $A_{F,\phi}(f)$. Consider the map
$h:\theta(Af)\oplus\theta(g)\rightarrow
\frac{\theta(Af)}{T\mathcal{A}_e Af}$ which maps
$(\xi_1,\xi_2)\mapsto \overline{\xi_1}$. Since
$h(T\mathcal{A}_e\{Af,g\})=\{0\}$ we consider the induced map
$$\overline{h}:\frac{\theta(Af)\oplus\theta(g)}{T\mathcal{A}_e\{Af,g\}}\rightarrow
\frac{\theta(Af)}{T\mathcal{A}_e Af}$$ where
$\overline{h}(\overline{(\xi_1,\xi_2)})=0$ if and only if
$h(\xi_1,\xi_2)=0$.

Obviously, any element in $\ker(\overline{h})$ can be taken to the
form $\overline{(0,\xi_2)}$. Now, if $\eta\in Lift(Af)$, then there
exists $\rho\in \theta_{n+1}$ such that $dAf\circ\rho=\eta\circ Af$.
Then, for any $\delta\in \theta_{n+1}$
$$\overline{(0,\xi_2)}=\overline{(0,\xi_2)-[(dAf\circ\rho,dg\circ\delta)+(\eta\circ
Af,\eta\circ g)]}=\overline{(0,\xi_2-dg\circ\delta-\eta\circ g)},$$
which means that
$$ker(\overline{h})\cong\frac{\theta(g)}{tg(\theta_{n+1})+wg(Lift(Af))}.$$

So we have the short exact sequence:

\begin{equation*}
\begin{CD}
0@>>> \frac{\theta(g)}{tg(\theta_{n+1})+wg(Lift(Af))} @>>> N\mathcal
A_e(\{Af,g\}) @>{\overline{h}}>> N\mathcal A_e(Af) @>>> 0
\end{CD}
\end{equation*}

Then
$$\mathcal{A}_e-cod(\{Af,g\})=\mathcal{A}_e-cod(Af)+dim_{\mathbb K}
\frac{\theta(g)}{tg(\theta_{n+1})+wg(Lift(Af))}.$$

Note that $tg(\theta_{n+1})=\sum_{l=1}^p\mathcal
O_{n+1}\frac{\partial}{\partial X_l}+\sum_{j=p+1}^{n+1}\mathcal
O_{n+1}v_j\frac{\partial}{\partial Z}$, therefore, by projection to
the last component we have that
$$\frac{\theta(g)}{tg(\theta_{n+1})\!+\!wg(Lift(Af))}\!\cong\!
\frac{\mathcal O_{n+1}}{\langle
v_{p+1},\ldots,v_{n+1}\rangle\!+\!dZ(wg(Lift(Af)))}\!\cong\!
\frac{\mathcal O_{p}}{dZ(i^*(Lift(Af)))}.$$

Since $F$ is stable, by Damon's theorem
$\mathcal{A}_e$-cod$(f)=dim_{\mathbb
K}\frac{\theta(i)}{ti(\theta_p)+i^*(Lift(F))}$. As $i(X)=(X,0)$, it
follows that $\frac{\theta(i)}{ti(\theta_p)+i^*(Lift(F))}$ is
isomorphic to $\frac{\mathcal O_p}{dZ(i^*(Lift(F)))}$ by projection
to the last component.

Therefore, $dim_{\mathbb K}\frac{\mathcal O_p}{dZ(i^*(Lift(F)))}\leq
dim_{\mathbb K}\frac{\mathcal O_p}{dZ(i^*(Lift(Af)))}$ gives
$$\mathcal{A}_e-cod(\{Af,g\})\geq
\mathcal{A}_e-cod(Af)+\mathcal{A}_e-cod(f)\geq$$
$$\mathcal{A}_e-cod(f)\tau(\phi)+\mathcal{A}_e-cod(f)=\mathcal{A}_e-cod(f)(\tau(\phi)+1)$$and
equality is reached when $\phi$ is quasi-homogeneous and $\langle
dZ(i^*(Lift(Af)))\rangle =\langle dZ(i^*(Lift(F)))\rangle$.

ii) The 1-parameter unfolding $\{A_{F,\phi'}(f),G\}$ where
$\phi'(z,\mu)=\phi(z)+\mu$ and $G(X,v,\mu)=(X,\Sigma_{j=p+1}^{n+1}
v_j^2,\mu)$ is stable since each branch is stable and the analytic
strata have regular intersection.
\end{proof}

\begin{rem}
The exact sequence in the proof above remains exact when we replace
$Af$ by any finitely determined map-germ $H$ so we can deduce that
if $\mathcal M_p\subseteq dZ(i^*(Lift(H)))$, then $\mathcal
A_e-cod(\{H,g\})\leq\mathcal A_e-cod(H)+1$.
\end{rem}


\begin{ex}
\begin{enumerate}
\item[i)] Consider the family of augmentations of $f(x)=x^3$, $A^lf(x,z)=(x^3+z^lx,z)$, of $\mathcal A_e$-codimension $l-1$, where $\phi(z)=z^l$ is the augmenting function. We calculate $Lift(A^lf)=\langle 3lX\frac{\partial}{\partial X}+2Z\frac{\partial}{\partial Z},-2lZ^{3l-1}\frac{\partial}{\partial X}+9X\frac{\partial}{\partial Z}\rangle$. So the bigerm
\begin{equation}
\begin{cases}
(x^3+z^lx,z)\\
(x,z^2)
\end{cases}
\end{equation}
has codimension $l$.

\item[ii)] Let $f_l(x,y)=(x^3+y^lx,y)$ and $F_l(x,y,z)=(x^3+y^lx+zx,y,z)$ with augmentations $A^mF_l(x,y,z)=(x^3+y^lx+z^mx,y,z)$ of codimension $(l-1)(m-1)$. The defining equation of the discriminant is $27X^2+4(Y^l+Z^m)^3=0$. For any $m\geq1$, $\frac{\mathcal O_p}{dZ(i^*(Lift(Af)))}=\langle 1,Y,\ldots,Y^{l-2}\rangle$, so the codimension of the bigerm $\{A^mF_l,g\}$ is $(l-1)(m-1)+(l-1)=(l-1)m$.

\item[iii)] The augmentation $Af(x,y,z)=(x^4+yx^2+y^2x+z^2x,y,z)$ of $f(x,y)=(x^4+yx^2+y^2x,y)$ has codimension 2. Calculations using the computer package Singular show that $\frac{\mathcal O_p}{dZ(i^*(Lift(Af)))}=\langle 1,Y\rangle$, so $\mathcal A_e-cod(\{(Af,g)\})=2+2=4$.
\end{enumerate}
\end{ex}

\begin{teo}
Up to $\mathcal A$-equivalence, if $\mathbb K=\C$ and $\mathcal
A_e-\cod(f)=1$, the multigerm $\{Af,g\}$ is independent of the
choice of miniversal unfolding $F$ of $f$.
\end{teo}
\begin{proof}
Let $Af$ and $A'f$ be augmentations coming from different miniversal
unfoldings $F$ and $F'$ of $f$ which are $\mathcal A$-equivalent as
unfoldings. By \cite{robertamond}, $Af\sim_{\mathcal A}A'f$ via
diffeomorphisms $\varphi:(\C^p\times\C,0)\rightarrow
(\C^p\times\C,0)$ and $\psi:(\C^n\times\C,S\times\{0\})\rightarrow
(\C^n\times\C,S\times\{0\})$ such that $$\varphi\circ
Af=A'f\circ\psi$$ where
$\varphi(X,v)=(\varphi_1(X,v),\ldots,\varphi_p(X,v),\varphi_{p+1}(v))$.
To prove $\{Af,g\}\sim_{\mathcal A}\{A'f,g\}$ we need a
diffeomorphism $\theta$ such that $\varphi\circ g=g\circ\theta$.
This is satisfied with
$\theta(X,v)=(\theta_1(X,v),\ldots,\theta_{n+1}(X,v))$ where
$\theta_j(X,v)=\varphi_j\circ g(X,v)$ for $j\leq p$ and
$\sum_{i=p+1}^{n+1}(\theta_{i}(X,v))^2=\varphi_{p+1}(\sum_{i=p+1}^{n+1}v_i^2)$.
\end{proof}

\begin{rem}
$\{Af,g\}$ is still independent of the choice of miniversal
unfolding $F$ of $f$ when $f$ is of codimension higher than 1 if we
suppose that $Af\sim_{\mathcal A}A'f$ through a diffeomorphism of
the type
$\varphi(X,v)=(\varphi_1(X,v),\ldots,\varphi_p(X,v),\varphi_{p+1}(v))$.
\end{rem}

\subsection{Generalised concatenations}

Another type of operation is a concatenation with a stable germ
which generalises the monic concatenation defined earlier. It also
includes the binary concatenation as a particular case.

\begin{definition}  (\cite{robertamond})
Given germs $f_0:(\C^m,S)\rightarrow (\C^a,0)$ and
$g_0:(\C^l,T)\rightarrow (\C^b,0)$ with 1-parameter stable
unfoldings $F(y,s)=(f_s(y),s)$ and $G(x,s)=(g_s(x),s)$, the
multigerm $h$ with $|S|+|T|$ branches defined by

\begin{equation}
\begin{cases}
(X,y,u)\mapsto (X,f_u(y),u)\\
(x,Y,u)\mapsto (g_{u}(x),Y,u)
\end{cases}
\end{equation}
is called a binary concatenation of $f_0$ and $g_0$.
\end{definition}

\begin{definition}
Let $f:(\mathbb K^{n-s},S)\rightarrow (\mathbb K^{p-s},0)$, $s<p$,
be of finite $\mathcal{A}_e$-codimension and let $F:(\mathbb
K^{n},S\times \{0\})\rightarrow (\mathbb K^{p},0)$ be a
$s$-parameter stable unfolding of $f$ with
$$F(x_1,\ldots,x_n)=(F_1(x_1,\ldots,x_n),\ldots,F_{p-s}(x_1,\ldots,x_n),x_{n-s+1},\ldots,x_n),$$
where $F_i(x_1,\ldots,x_{n-s},0,\ldots,0)=f_i(x_1,\ldots,x_{n-s})$.
Suppose that $\overline g:(\mathbb K^{n-p+s},T)\rightarrow (\mathbb
K^{s},0)$ is stable. Then the multigerm $\{F,g\}$ is a generalised
concatenation of $f$ with $g$, where $g=Id_{\mathbb K^{p-s}}\times
\overline g$.
\end{definition}

Observe that with this definition, $\dim\widetilde\tau(g)\geq
p-s\geq 1$. If $g$ is a monogerm and $\dim\widetilde\tau(g)=p-s$, it
is of the form
$$g(x_1,\ldots,x_n)=(x_1,\ldots,x_{p-s},g_{p-s+1}(x_{p-s+1},\ldots,x_n),\ldots,g_{p}(x_{p-s+1},\ldots,x_n)).$$
Therefore, the definition implies that $F\pitchfork
\widetilde\tau(g)$.

\begin{rem}

\begin{enumerate}
\item[i)] The monic concatenation is recovered by taking $s=1$ and $g_p(x_p,\ldots,x_n)=\Sigma_{i=p}^{n}x_i^2$ (or $g_p=0$ when $n=p-1$).

\item[ii)] A binary concatenation $h=\{F,G\}$

\begin{equation}
\begin{cases}
(X,y,u)\mapsto (f_u(y),u,X)\\
(x,Y,u)\mapsto (Y,u,g_{u}(x))
\end{cases}
\end{equation}
is also a generalised concatenation where $p=a+1+b$, $n=b+m+1=l+a+1$
and $s=b+1$. In fact, the first branch is a $b+1$-parameter stable
unfolding of an $f_0$ and $\widetilde\tau((u,g_u(x)))=\{0\}$ when
$g_0$ is not stable.
\end{enumerate}
\end{rem}

In general this operation is difficult to control, so we study
particular cases where the stable germ $g$ is given, for example the
equidimensional case and where $\overline g$ is a cusp and
$\dim\widetilde\tau(g)=p-2$.

\begin{definition}
Consider $f:(\mathbb K^{n-2},S)\rightarrow (\mathbb K^{n-2},0)$ with
$n\geq 3$, $F(x,\lambda)=(f_{\lambda}(x),\lambda)$ a 2-parameter
stable unfolding of $f$ and
$$g(x_1,\ldots,x_{n-2},y,z)=(x_1,\ldots,x_{n-2},y,z^3+yz)$$ being a
suspension of a cusp. We call the multigerm $\{F,g\}$ the cuspidal
concatenation of $f$.
\end{definition}

In general $\mathcal{A}_e$-cod$(\{F,g\})$ depends on the choice of
2-parameter stable unfolding, so we give a recipe to calculate the
$\mathcal A_e$-codimension.

\begin{teo}\label{cuspconc}
Let $f:(\mathbb K^{n-2},S)\rightarrow (\mathbb K^{n-2},0)$ with
$n\geq 3$ and $\{F,g\}$ the cuspidal concatenation of $f$, then
$$\mathcal{A}_e-cod(\{F,g\})=\dim_{\mathbb
K}\frac{\mathcal{O}_{n-1}}{\{\xi:\xi=-z\eta_{n-1}(x,-3z^2,-2z^3)+\eta_{n}(x,-3z^2,-2z^3)\}},$$
where $\eta_{n-1}$ and $\eta_{n}$ are the last two components of
vector fields in $Lift(F)$.
\end{teo}
\begin{proof}

Consider the exact sequence (see the proof of \ref{augconc})

\begin{equation*}
\begin{CD}
0@>>> \frac{\theta(g)}{tg(\theta_{n})+wg(Lift(F))} @>>> N\mathcal
A_e(\{F,g\}) @>>> N\mathcal A_e(F) @>>> 0.
\end{CD}
\end{equation*}
 Since $F$ is stable, $\dim_{\mathbb K} N\mathcal A_e(F)=0$, hence $\mathcal A_e$-cod$(\{F,g\})$ is equal to the dimension of $\frac{\theta(g)}{tg(\theta_{n})+wg(Lift(F))}$. By projecting to the last two components, this space is isomorphic to $\frac{\mathcal O_{n}\oplus\mathcal O_{n}}{T}$ where $$T=\{\left(\begin{array}{cc}1 & 0\cr z & 3z^2+y\cr\end{array}\right)\left(\begin{array}{c} v_{n-1}\cr v_{n}\cr\end{array}\right): v_{n-1},v_{n}\in \mathcal O_{n}\} +d(Y,Z)(wg(Lift(F))),$$ and $d(Y,Z)$ represents the last two components of $wg(Lift(F))$.

Let $$T_0=\{\xi: (0,\xi)\in
T\}=$$$$=\{\xi:\xi=-z\eta_{n-1}(x,y,z^3+yz)+(3z^2+y)v_{n}(x,y,z)+\eta_{n}(x,y,z^3+yz)\},$$
where $\eta=(\eta_1,\ldots,\eta_{n})\in Lift(F).$

Let $(g_{n-1},g_{n})$ be the last two components of $g$ and let
$$T_1=tg_{n-1}(\mathcal O_{n})+dY(wg(Lift(F))).$$

The following sequence is exact

\begin{equation*}
\begin{CD}
0@>>> \frac{\mathcal O_{n}}{T_0} @>{i^*}>> \frac{\mathcal
O_{n}\oplus\mathcal O_{n}}{T} @>{\pi^*}>>
\frac{\theta(g_{n-1})}{T_1} @>>> 0
\end{CD}
\end{equation*}
where $i$ is the inclusion and $\pi$ is the projection. The proof of
the exactness is analogous to the proof of Proposition 2.1 in
\cite{mancini}. Since $g_{n-1}$ is a submersion, $\frac{\mathcal
O_{n}\oplus\mathcal O_{n}}{T}$ and $\frac{\mathcal O_{n}}{T_0}$ are
isomorphic. The result follows from the fact that
$$\frac{\mathcal{O}_{n}}{T_0}\cong\frac{\mathcal{O}_{n-1}}{\{\xi:\xi=-z\eta_{n-1}(x,-3z^2,-2z^3)+\eta_{n}(x,-3z^2,-2z^3)\}}.$$
\end{proof}

\begin{prop}\label{a2an}
Let $f(x_1,\ldots,x_{n-2})=(x_1^{n+1}+x_2x_1+\ldots
+x_{n-2}x_1^{n-3},x_2,\ldots,x_{n-2})$ and let
$$F(x_1,\ldots,x_{n-2},y,z)=(x_1^{n+1}+x_2x_1+\ldots
+x_{n-2}x_1^{n-3}+yx_1^{n-2}+zx_1^{n-1},x_2,\ldots,x_{n-2},y,z),$$
the $A_n$ singularity. Then the cuspidal concatenation of $f$ has
$\mathcal{A}_e$-cod$(\{F,g\})=n$.
\end{prop}
\begin{proof}
We denote by
$T_{\Sigma}=\{\xi:\xi=-z\eta_{n-1}(x,-3z^2,-2z^3)+\eta_{n}(x,-3z^2,-2z^3)\}$.
From Theorem \ref{cuspconc}, $\mathcal A_e-cod(\{F,g\})=dim_{\mathbb
K}\frac{\mathcal O_{n-1}}{T_{\Sigma}}$. From \cite{bruce} it follows
that the linear part of the generators of $Lift(F)$ is
$$\{(n+1)X_1\frac{\partial}{\partial
Z},(n+1)X_1\frac{\partial}{\partial Y}+nX_2\frac{\partial}{\partial
Z},\ldots,$$ $$(n+1)X_1\frac{\partial}{\partial X_2}+\ldots
+4X_{n-2}\frac{\partial}{\partial Y}+3Y\frac{\partial}{\partial
Z},$$ $$(n+1)X_1\frac{\partial}{\partial X_1}+\ldots
+3Y\frac{\partial}{\partial Y}+2Z\frac{\partial}{\partial Z}\}.$$
The corresponding elements of the generators of $T_{\Sigma}$ and the
corresponding relations modulo $\mathcal M_{(x,z)}^3$ are

\begin{equation}\label{bru}
\begin{cases}
(n+1)x_1\\
-(n+1)zx_1+nx_2\\
\ldots\\
-5zx_{n-3}+4x_{n-2}\\
-4zx_{n-2}-9z^2\\
5z^3
\end{cases}
\Longrightarrow
\begin{cases}
(1)\,\,\,\, x_1\equiv 0\\
(2)\,\,\,\, x_2\equiv \frac{n+1}{n}zx_1\\
\ldots\\
(n-2)\,\,\,\, x_{n-2}\equiv \frac{5}{4}zx_{n-3}\\
(n-1)\,\,\,\, z^2\equiv -\frac{4}{9}zx_{n-2}
\end{cases}
\end{equation}

Note that $T_{\Sigma}$ is a $\mathcal O_{n-1}-$module via
$\varphi^*$ where $\varphi(x,z)=(x,-3z^2,-2z^3).$ It is clear that
we can obtain any element in $T_{\Sigma}$ of the type
$x^{\alpha}z^{\beta}$ for all $\alpha>0$ and $\beta=0$ or $\beta\geq
2$ modulo $\mathcal M^{\alpha+\beta+1}$ and of the type $z^{\beta}$
for all $\beta\geq 3$. It is also clear that constants and $z$ are
not in $T_{\Sigma}$. Lets see what other terms are in $T_{\Sigma}$:

In step I, we take relation (1) and multiply it by $x_j$ for all
$j=1,\ldots,n-2$ and by $z^2$ to obtain relations $x_1xj\equiv 0$
and $x_1z^2\equiv 0$ modulo $\mathcal M^4$. Now multiply relations
(2) to $(n-1)$ by $x_1$ and use the relations just obtained to get
relations $zx_1x_{j-1}\equiv 0$ modulo $\mathcal M^4$ for all
$j=2,\ldots,n-1$.

In step II, we take relation (2) and multiply it by $x_j$ for all
$j=2,\ldots,n-2$ and by $z^2$ to obtain relations $x_2xj\equiv
zx_1x_j$ ($\equiv 0$ by step I) and $x_2z^2\equiv 0$ modulo
$\mathcal M^4$. Now multiply relations (3) to $(n-1)$ by $x_2$ and
use the relations just obtained to get relations $zx_2x_{j-1}\equiv
0$ modulo $\mathcal M^4$ for all $j=3,\ldots,n-1$.

We go on until finally we have that $\mathcal M^3\subseteq
T_{\Sigma}+\mathcal M^4$. In the same way we prove that $\mathcal
M^k\subseteq T_{\Sigma}+\mathcal M^{k+1}$ for $k\geq4$, so we have
that $\mathcal M^3\subseteq T_{\Sigma}+\mathcal M^{k}$ for all
$k\geq4$. Therefore, $\mathcal M_3\subseteq T_{\Sigma}.$ A basis for
$\frac{\mathcal O_{n-1}}{T_{\Sigma}}$ is $1,z,x_1z,\ldots,x_{n-2}z$,
which proves the result.
\end{proof}

\begin{teo}\label{codn}
Let $F:(\K^n,S)\rightarrow(\K^n,0)$ be a stable (multi)germ with
$\widetilde\tau(F)=0$ and
$g(x_1,\ldots,x_{n-2},y,z)=(x_1,\ldots,x_{n-2},y,z^3+yz)$, then
$\mathcal{A}_e$-cod$(\{F,g\})\geq n$ and is equal to $n$ when $g$ is
transversal to the limits of the tangent spaces to the strata of the
stratification by stable types of the discriminant of $F$ and
$F\pitchfork \widetilde\tau(g)$.
\end{teo}
\begin{proof}
The fact that $F$ is stable, has corank 1 and $\widetilde\tau(F)=0$
implies that $F$ is a singularity of type $A_{k_1}\ldots A_{k_r}$
where $k_1+\ldots +k_r=n$. In the case of lowest codimension, $g$ is
transversal to the limits of the tangent spaces to the strata of the
discriminant of $F$ and $F\pitchfork \widetilde\tau(g)$. In this
case, $A_{k_i}$ considered from $\mathbb K^n$ to $\mathbb K^n$ is a
$n-k_i$-prism on $A_{k_i}$ considered from $\mathbb K^{k_i}$ to
$\mathbb K^{k_i}$. Then $Lift(A_{k_i})$ is as in Proposition
\ref{a2an} together with $\frac{\partial}{\partial X_j}$ with $X_j$
varying on the remaining $n-k_i$ variables. Since $A_{k_1}\ldots
A_{k_r}$ is stable and all of the $A_{k_i}$ are stable, the
$Lift(F)$ is the intersection of the $Lift(A_{k_i})$ and so we can
see it as a diagonal block matrix where each block represents the
$Lift(A_{k_i})$. Following the proof of Proposition \ref{a2an} we
prove that the codimension of $\{F,g\}$ in this case is exactly $n$.
In other cases it is greater than or equal to $n$.
\end{proof}

The following examples illustrate how the cuspidal concatenation
depends on the choice of stable unfolding:

\begin{ex}
\begin{enumerate}
\item[i)] Let $f(x)=x^4$ and choose $F(x,y,z)=(x^4+yx+zx^2,y,z)$ (the swallowtail). Concatenating with a cuspidal edge we obtain the codimension 3 bigerm $A_2A_3$:

\begin{equation}
\begin{cases}
(x^4+yx+zx^2,y,z)\\
(x,y,z^3+yz)
\end{cases}
\end{equation}
Notice that this bigerm could not be obtained by any of the other
operations in the literature up to now. It can be seen that it is
not simple and the stratum codimension is 2.

Now, we let $F(x,y,z)=(x^4+yx^2+y^2x+zx,y,z)$, then the bigerm
$\{F,g\}$ is not $\mathcal A$-equivalent to $A_2A_3$ above. In
fact,$$Lift(F)=\langle 4X\frac{\partial}{\partial
X}+2Y\frac{\partial}{\partial Y}+(Y^2-3Z)\frac{\partial}{\partial
Z},$$ $$(Y^3+ZY)\frac{\partial}{\partial
X}+(-6Y^2-6Z)\frac{\partial}{\partial
Y}+(8X+2Y^2+12Y^3+12YZ)\frac{\partial}{\partial Z},$$
$$(-16XY-9Z^2-18Y^2Z-9Y^4)\frac{\partial}{\partial
X}+(48X+4Y^2)\frac{\partial}{\partial
Y}+(-96XY+12YZ+4Y^3)\frac{\partial}{\partial Z}\rangle,$$ and
calculations show that $\mathcal A_e-cod(\{F,g\})=4$.

\item[ii)] Consider $f(x)=x^3$ and the family of 2-parameter stable unfoldings $F_l(x,y,z)=(x^3+y^lx+zx,y,z)$. In this case $$Lift(F_l)=\langle 3lX\frac{\partial}{\partial X}+2Y\frac{\partial}{\partial Y}+2lZ\frac{\partial}{\partial Z},\frac{\partial}{\partial Y}+lY^{l-1}\frac{\partial}{\partial Z},2(Z+Y^l)^2\frac{\partial}{\partial X}-9X\frac{\partial}{\partial Z}\rangle.$$

Calculations show that $\mathcal A_e-cod(\{F_l,g\})=2$ for $l\geq
2$.

For $l=1$ we have the codimension 1 bigerm $A_2^2$ (which is a
binary concatenation, see \cite{robertamond}). For $l\geq 2$ we have
the codimension 2 bigerm $T_{22}^1$ (see \cite{mio}). It is
interesting to mention that in spite of the cuspidal edge in $F_l$
having a more degenerate contact with the limiting tangent plane to
$g$ as $l$ grows, this does not affect the codimension. Compare this
with example 4.5 ii) when $m=1$, where the codimension does
increase.

\item[iii)] Let $f(x)=\{x^2,x^3\}$, and $F_1$ and $F_2$ two different 2-parameter stable unfoldings:

    \begin{equation}
F_1=
\begin{cases}
(x^2+y+z,y,z)\\
(x^3+xy,y,z)
\end{cases}
\text{and   } F_2=
\begin{cases}
(x^2,y,z)\\
(x^3+xy+z,y,z)
\end{cases}
\end{equation}

One would expect $\{F_1,g\}\sim_{\mathcal A}\{F_2,g\}$ since they
are two cuspidal edges and a fold plane with pairwise transversal
branches. However, $\mathcal A_e-cod(\{F_1,g\})=3$ while $\{F_2,g\}$
is not finitely determined because it has a curve of triple points
in the image along $(0,-3t^2,2t^3)$.

\end{enumerate}
\end{ex}

Another type of generalised concatenation is to concatenate with two
fold hypersurfaces (in the equidimensional case, but the operation
can be defined for $n\neq p$ too).

\begin{definition}
Let $f:(\K^{n-2},S)\rightarrow(\K^{n-2},0)$ ($n\geq 3$) be a
finitely determined germ, $F(x,\lambda)=(f_{\lambda}(x),\lambda)$ a
2-parameter stable unfolding of $f$ and $g=\{g_1,g_2\}$ where
\begin{equation}
\begin{cases}
g_1(x_1,\ldots,x_{n-2},y,z)=(x_1,\ldots,x_{n-2},y,z^2)\\
g_2(x_1,\ldots,x_{n-2},y,z)=(x_1,\ldots,x_{n-2},y,z^2+y).
\end{cases}
\end{equation}
We call the multigerm $\{F,g\}$ the double fold concatenation of
$f$.
\end{definition}

As in the cuspidal concatenation, the $\mathcal A_e$-codimension of
$\{F,g\}$ depends on the choice of the 2-parameter stable unfolding.

\begin{teo}\label{doubleconc}
 Let $f:(\K^{n-2},S)\rightarrow(\K^{n-2},0)$ with $n\geq 3$ and let $\{F,g\}$ be the double fold concatenation of $f$, then $$\mathcal{A}_e-\cod(\{F,g\})=\mathcal{A}_e-\cod(\{F,g_1\})+ \dim_{\mathbb K}\frac{\mathcal{O}_{n-1}}{\{\xi:\xi=-\eta_{n-1}(x,y,y)+\eta_{n}(x,y,y)\}},$$ where $\eta_{n-1},\eta_n$ are the last two components of vector fields in $Lift(\{F,g_1\})$.
\end{teo}
\begin{proof}
The proof is very similar to the one of the cuspidal concatenation.
Consider the following exact sequence

\begin{equation*}
\begin{CD}
0@>>>\!\frac{\theta(g_2)}{tg_2(\theta_{n})+wg_2(Lift(\{F,g_1\}))}\!@>>>
N\mathcal A_e(\{F,g\})\!@>>>\!N\mathcal A_e(\{F,g_1\})\!@>>>\!0
\end{CD}
\end{equation*}
Projecting to the last two components it follows that
$\frac{\theta(g_2)}{tg_2(\theta_{n})+wg_2(Lift(\{F,g_1\}))}$ is
isomorphic to $\frac{\mathcal O_{n}\oplus\mathcal O_{n}}{T}$ where
$$T=\{\left(\begin{array}{cc}1 & 0\cr 1 &
2z\cr\end{array}\right)\left(\begin{array}{c} v_{n-1}\cr
v_{n}\cr\end{array}\right): v_{n-1},v_{n}\in \mathcal O_{n}\}
+d(Y,Z)(wg_2(Lift(\{F,g_1\}))),$$ and $d(Y,Z)$ represents the last
two components of $wg_2(Lift(\{F,g_1\}))$.

Let $T_0=\{\xi: (0,\xi)\in T\}$ which is equal to
$$\{\xi:\xi=-\eta_{n-1}(x,y,z^2+y)+zv_n(x,y,z)+\eta_{n}(x,y,z^2+y)\},$$
where $\eta=(\eta_1,\ldots,\eta_{n})\in Lift(\{F,g_1\}).$ As in
Proposition \ref{cuspconc} we can show that $\frac{\mathcal
O_{n}\oplus\mathcal O_{n}}{T}$ is isomorphic to
$\frac{\mathcal{O}_{n}}{T_0}$ which is in turn isomorphic to
$\frac{\mathcal{O}_{n-1}}{\{\xi:\xi=-\eta_{n-1}(x,y,y)+\eta_{n}(x,y,y)\}}$.
\end{proof}

\begin{rem}
If there exists a finitely determined 1-parameter unfolding $F'$ of
$f$, $F':(\K^{n-1},S\times\{0\})\rightarrow (\K^{n-1},0)$,
$F'(x,y)=(f_y(x),y)$ such that $F(x,y,0)=(F'(x,y),0)$, then the
multigerm $\{F,g_1\}$ is a monic concatenation of $F'$. Therefore,
$\mathcal A_e-\cod(\{F,g_1\})=\mathcal A_e-\cod(F')$ by Theorem
\ref{codconc}.
\end{rem}

\begin{ex}
Consider the trigerm $f(x)=\{x^2,x^2,x^2\}$ and $F'$ and $F$ as
follows:

\begin{equation}
F'=
\begin{cases}
(x^2+y,y)\\
(x^2,y)\\
(x^2-y,y)
\end{cases}
\text{and  } F=
\begin{cases}
(x^2+y+z,y,z)\\
(x^2,y,z)\\
(x^2-y,y,z)
\end{cases}
\end{equation}

The generators of $Lift(\{F,g_1\})$ are
$$\{X\frac{\partial}{\partial X}+Y\frac{\partial}{\partial
Y}+Z\frac{\partial}{\partial Z},$$
$$(2XY+XZ)\frac{\partial}{\partial
X}+(3X^2-Y^2-2XZ-YZ)\frac{\partial}{\partial Y},$$
$$(2X^2-3XZ)\frac{\partial}{\partial
X}+(2XY+4XZ+YZ)\frac{\partial}{\partial
Y}+(Z^2-XZ)\frac{\partial}{\partial Z},$$
$$(2XY+XZ)\frac{\partial}{\partial
X}+(2Y^2+4XZ+5YZ)\frac{\partial}{\partial
Y}+(-3Z^2-6XZ)\frac{\partial}{\partial Z}\}.$$ Calculations show
that $dim_{\mathbb
K}\frac{\mathcal{O}_{p-1}}{\{\xi:\xi=-\eta_{p-1}(x,y,y)+\eta_{p}(x,y,y)\}}=3$.
Since $\mathcal A_e-\cod(\{F,g_1\})=\mathcal A_e-\cod(F')=1$, it
follows that the codimension of the quintuple point $\{F,g\}$ is 4.
\end{ex}

\begin{teo}\label{codndoublefold}
Let $F:(\K^n,S)\rightarrow(\K^n,0)$ be a stable (multi)germ with
$\widetilde\tau(F)=0$ and $g=\{g_1,g_2\}$ where
\begin{equation}
\begin{cases}
g_1(x_1,\ldots,x_{n-2},y,z)=(x_1,\ldots,x_{n-2},y,z^2)\\
g_2(x_1,\ldots,x_{n-2},y,z)=(x_1,\ldots,x_{n-2},y,z^2+y),
\end{cases}
\end{equation}
then $\mathcal{A}_e$-cod$(\{F,g\})\geq n$.
\end{teo}
\begin{proof}
If $F$ is a monogerm, we use the information about $Lift(F)$ in
Proposition \ref{a2an} and the fact that $Lift(g_1)=\langle
\frac{\partial}{\partial X_1},\ldots,\frac{\partial}{\partial
X_{n-2}},\frac{\partial}{\partial Y},Z\frac{\partial}{\partial
Z}\rangle$ to apply the formula in Theorem \ref{doubleconc}. If $F$
is a multigerm we repeat the argument in the proof of Theorem
\ref{codn}.
\end{proof}

A similar result for the case $n=p-1$ can be found in Remark
\ref{primmultmorse}.

\section{$\mathcal A_e$-codimension 2 multigerms}

The classification of primitive monogerms of $\mathcal
A_e$-codimension 2 for all pairs $(n,p)$ in the nice dimensions and
$p\leq n+1$ is already known. When $n\geq p$, the list of $\mathcal
A$-simple singularities was obtained by Goryunov in \cite{Goryunov}
(see also \cite{riegerruas}). When $p=n+1$, this classification was
recently given by Houston and Wik Atique in \cite{houstonroberta}.

We assume the known fact that when $p=1$, the only codimension 2
multigerms are a trigerm with 3 Morse functions and a bigerm with a
Morse function and an $A_2$ singularity. We also need the list of
codimension 2 multigerms when $p=2$. For $n=1$ we refer to
\cite{kolgushkin} and for $n\geq 2$ a list can be found in
\cite{Ohmoto}.

In this section we prove that all the simple minimal corank
$\mathcal A_e$-codimension 2 multigerms with more than one branch in
Mather's nice dimensions and $n\geq p-1$ can be obtained using the
operations of augmentation, augmentation and concatenation and
generalised concatenations starting from monogerms and some special
multigerms. By a multigerm with $k$ branches we mean a multigerm
with $k$ nonsubmersive branches. Too that all the results in this
section are stated for the complex case.

The section is organised as follows: We start by proving some
general results for any codimension 2 multigerm. Results \ref{aug2}
to \ref{corocod1} deal with augmentations and $\mathcal
A_e$-codimension 2 multigerms that have a branch which is an
$\mathcal A_e$-codimension 1 monogerm. Results \ref{decomp} to
\ref{trans} classify the codimension 2 primitive multigerms where
all the branches are stable. Finally, Theorem \ref{teofinal} is a
summary of all the results in this section.

\begin{prop}\label{cod2}
Let $h=\{h_1,\ldots,h_r\}:(\C^n,S)\rightarrow(\C^p,0)$ be a
multigerm of $\mathcal A_e$-codimension 2. Then, for any proper
subset $S'$ of $S$, the multigerm $h':(\C^n,S')\rightarrow(\C^p,0)$
has $\mathcal A_e$-codimension $\leq 1$. If $\mathcal
A_e-\cod(h')=1$ then $h=\{h',g\}$ where $g$ is prism on a Morse
function (when $n\geq p$) or an immersion (when $p=n+1$).
\end{prop}
\begin{proof}
Let $g:(\C^n,S\backslash S')\rightarrow(\C^p,0)$ be the germ such
that $h=\{h',g\}$. If $h'$ is stable, then $\mathcal
A_e-\cod(h')=0$. Otherwise, consider the exact sequence

\begin{equation*}
\begin{CD}
0@>>> \frac{\theta(g)}{tg(\theta_{n})+wg(Lift(h'))} @>>> N\mathcal
A_e(h) @>>> N\mathcal A_e(h') @>>> 0.
\end{CD}
\end{equation*}
Since $\widetilde\tau(h')=0$, then the dimension of
$\frac{\theta(g)}{tg(\theta_{n})+wg(Lift(h'))}$ is at least 1 and
since $\dim_{\mathbb C} N\mathcal A_e(h)=2$, it follows that the
codimension of $h'$ has to be less than or equal to 1.

Now suppose that $\mathcal A_e-\cod(h')=1$, then the dimension of
$\frac{\theta(g)}{tg(\theta_{n})+wg(Lift(h'))}$ would be exactly 1.
Since $\widetilde\tau(h')=\{0\}$, $Lift(h')$ does not have any
constants in any entry and so
$1=\dim\frac{\theta(g)}{tg(\theta_{n})+wg(Lift(h'))}\geq\dim\frac{\theta(g)}{tg(\theta_{n})+g^*(\mathcal
M_p)\theta_p}=\mathcal K_e-cod(g)$. This implies first that $g$ is a
monogerm since the $\mathcal K_e-$codimension of the simplest
possible bigerms (namely two transversal folds or immersions) is 2.
Furthermore, $\dim\frac{\theta(g)}{tg(\theta_{n})+g^*(\mathcal
M_p)\theta_p+wg(\theta_p)}=0$ and so $g$ is stable and therefore it
is a prism on a Morse function or an immersion.
\end{proof}

\begin{coro}
Let $h=\{h_1,\ldots,h_r\}:(\C^n,S)\rightarrow(\C^p,0)$ ($p>1$) be a
multigerm of $\mathcal A_e$-codimension 2. Then
\begin{enumerate}
\item[i)] If $r\geq 3$ then $h_i$ is stable for every $i\in \{1,\ldots,r\}$.

\item[ii)] If $h_i$ is stable for every $i\in \{1,\ldots,r\}$, then $h=\{f,g\}$ where both $f$ and $g$ are stable.
\end{enumerate}
\end{coro}
\begin{proof}
i) Follows directly from Proposition \ref{cod2}.

ii) From \cite{robertamond} we know that a codimension 1 multigerm
only has stable branches.

Suppose that there does not exist a partition $h=\{f,g\}$ with $f,g$
stable. It follows from Proposition \ref{cod2} that $h_i$ is a prism
on a Morse function or an immersion for all $i=1,\ldots,r$. By
hypothesis, $h'=\{h_{i_1},\ldots,h_{i_{r-1}}\}$ has $\mathcal
A_e$-codimension 1, and by \cite{robertamond} any $r-2$ branches of
$h'$ form a stable multigerm. It follows that any two branches of
$h$ constitute an $\mathcal A_e$-codimension 1 bigerm. Therefore, if
$r\geq 4$ we have a decomposition of $h$ where at least two bigerms
have $\mathcal A_e$-codimension 1, which contradicts Proposition
\ref{cod2}. When $r=3$, in a similar way we can prove that all
branches are prisms on Morse functions or immersions and, further
more, that any two branches form a codimension 1 germ, that is, they
are tangent. This implies a triple tangency. Direct calculations
show that the codimension of such a trigerm is greater than two
except for when $p=1$. The case $r=2$ is trivial.
\end{proof}

\begin{rem}\label{3morse}
When $p=1$, a trigerm of three Morse functions has codimension 2 and
cannot be separated into two stable germs.
\end{rem}

\begin{prop}\label{almreg2}
Let $h=\{f,g\}$ be a multigerm of $\mathcal A_e$-codimension 2. Then
\begin{enumerate}
\item[i)] If $f$ is a monogerm of $\mathcal A_e$-codimension 1, then $\widetilde\tau(f)$ and $\widetilde\tau(g)$ have almost regular intersection.

\item[ii)] Suppose $f$ and $g$ are stable, then $\widetilde\tau(f)$ and $\widetilde\tau(g)$ have almost regular intersection ($\cod\widetilde\tau(f)+\cod\widetilde\tau(g)-\cod(\widetilde\tau(f)\cap\widetilde\tau(g))=1$) if and only if $h$ admits a 1-parameter stable unfolding. They have almost regular intersection of order 2 otherwise ($\cod\widetilde\tau(f)+\cod\widetilde\tau(g)-\cod(\widetilde\tau(f)\cap\widetilde\tau(g))=2$).
\end{enumerate}
\end{prop}
\begin{proof}
i) If $f$ is of codimension 1, then $\dim\widetilde\tau(f)=0$,
therefore $\cod(\widetilde\tau(f)\cap\widetilde\tau(g))=p$. Since in
this case, by Proposition \ref{cod2}, $g$ is a prism on a Morse
function or an immersion, it follows that $\cod\widetilde\tau(g)=1$
and so $\cod\widetilde\tau(f)+\cod\widetilde\tau(g)=p+1$.

ii) From Proposition \ref{almreg} it follows that $h=\{f,g\}$ admits
a 1-parameter stable unfolding if and only if $\widetilde\tau(f)$
and $\widetilde\tau(g)$ have almost regular intersection.

Repeating the proof of Proposition \ref{almreg} replacing the
1-parameter stable unfolding by a 2-parameter versal unfolding we
obtain that if $h$ does not admit a 1-parameter stable unfolding,
then the analytic strata of $f$ and $g$ have almost regular
intersection of order 2.
\end{proof}

Given $f:(\C^n,0)\rightarrow(\C^p,0)$ with $n\leq p$ let
$m_0(f)=\dim_{\mathbb C}\frac{\mathcal O_n}{f^*(\mathcal M_p)}$ be
the multiplicity of $f$. If $n>p$, and $f$ is simple of corank 1, by
\cite{riegerruas} $f$ can be seen as a germ $f_0$ with $n=p$ plus
quadratic terms in the remaining variables so by multiplicity we
mean $m_0(f_0)$.

In the case $n\geq p$ a stable germ can have multiplicity at most
$p+1$ (\cite{mather}). This means that a germ with multiplicity
$p+2$ cannot be an augmentation, since that would imply that there
is a stable germ in the same dimensions with that same multiplicity.
When $p=n+1$, a stable germ can have multiplicity at most
$[\frac{n}{2}]+1$ (\cite{morin}).

\begin{prop}\label{aug2}
Let $h=\{f,g\}$ be a multigerm of $\mathcal A_e$-codimension 2 which
admits a 1-parameter stable unfolding. Then
\begin{enumerate}
\item[i)] If $f$ and $g$ are stable, then $h$ is an augmentation if and only if $\cod\widetilde\tau(f)+\cod\widetilde\tau(g)\leq p$.

\item[ii)] If $f$ is a monogerm of codimension 1, then $h$ is an augmentation if and only if $m_0(f)\leq p$ in the case $n\geq p$ ($m_0(f)\leq [\frac{n}{2}]$ in the case $p=n+1$).
\end{enumerate}
\end{prop}
\begin{proof}
i) See Corollary \ref{aug}.

ii) First suppose that $n\geq p$, $h$ is an augmentation and $f$ is
an augmentation of multiplicity greater than or equal to $p+1$. This
means that there exist a map
$h_0=\{f_0,g_0\}:(\C^{n-1},S)\rightarrow (\C^{p-1},0)$, admitting a
1-parameter stable unfolding, and that $f_0$ satisfies
$m_0(f_0)=m_0(f)\geq(p-1)+2$. Then $\mathcal A_e-\cod(f_0)=1$ and
therefore $g_0$ is a prism on a Morse function or an immersion. Then
we get a contradiction, as any 1-parameter stable unfolding $F$ of
$f_0$ has $\widetilde\tau(F)=0$ (for $f_0$ is primitive), and
therefore $\widetilde\tau(F)$ cannot be transversal to the analytic
stratum of any 1-parameter unfolding of $g_0$. It follows that
$m_0(f)\leq p$.

Now suppose that $f$ is an augmentation of multiplicity $k\leq p$ of
an $f_0$. Then $f_0$ admits a 1-parameter stable unfolding $F_0$
with $\dim\widetilde\tau(F_0)\geq 1$. We can choose a prism on a
Morse function or an immersion $g_0:(\C^{n-1},0)\rightarrow
(\C^{p-1},0)$ such that $\widetilde\tau(g_0\times id_t)$ (which has
dimension $p-1$) is transversal to $\widetilde\tau(F_0)$. The bigerm
$\{f,g_0\times id_t\}$ is an augmentation of $h_0=\{f_0,g_0\}$, and
since it has codimension 2 and is simple, it is $\mathcal
A$-equivalent to $h$, which is therefore an augmentation.

The same proof is valid for the case $p=n+1$ but with $m_0(f)\leq
[\frac{n}{2}]$ instead of $m_0(f)\leq p$.
\end{proof}

\begin{rem}
In the setting of Proposition \ref{aug2}, part ii), if $m_0(f)=p+1$
(resp. $m_0(f)=[\frac{n}{2}]+1$), then $h$ is clearly an
augmentation and concatenation in the case $n\geq p$ (resp. in the
case $p=n+1$).
\end{rem}

\begin{lem}\label{primngeqp}
If $f$ is a $\mathcal A_e$-codimension 1 primitive monogerm and
$n\geq p$, then, besides the Euler vector field, the components of
vector fields in $Lift(f)$ are in $\mathcal M_p^2$.
\end{lem}
\begin{proof}
Let $F$ be a mini-versal unfolding of $f$. The discriminant of $f$
is a section of the discriminant of $F$. The $Lift(f)$ is obtained
from $Lift(F)\cap Lift(g)$ where $g$ is a fold whose discriminant
gives the section, in the following way.

Consider first the equidimensional case. Then $f:(\C^n,0)\rightarrow
(\C^n,0)$ has a normal form
$f(x_1,\ldots,x_n)=(x_1^{n+2}+x_2x_1+\ldots+x_nx_1^{n-1},x_2,\ldots,x_n)$.
The section $x_{n+1}=0$ of the discriminant of $F$ is the
discriminant of $f$. The $Lift(f)$ is therefore the projection to
the first $n$ components of $Lift(F)\cap Lift(g)$ and substituting
$X_{n+1}$ by 0, where
$g(x_1,\ldots,x_{n+1})=(x_1,\ldots,x_n,x_{n+1}^2)$ and
$(X_1,\ldots,X_{n+1})$ are the target coordinates. Since
$Lift(g)=\langle\frac{\partial}{\partial
X_1},\ldots,\frac{\partial}{\partial
X_n},X_{n+1}\frac{\partial}{\partial X_{n+1}}\rangle$, and the
linear part of $Lift(F)$ is as given in Proposition \ref{a2an}, it
can be seen that the only vector field with linear terms is the
Euler vector field.

For the cases from $\C^{n+k}$ to $\C^n$, the normal forms for
primitive codimension 1 germs are $(x_1,\ldots,x_{n+k})\mapsto
(x_1^{n+2}+x_2x_1+\ldots+x_nx_1^{n-1}+\sum_{i=n+1}^{i=k}x_i^2,x_2,\ldots,x_n)$
and similar arguments prove the same result.
\end{proof}

\begin{lem}\label{liftprim}
If $f$ is a $\mathcal A_e$-codimension 1 primitive monogerm and
$n=p-1$, then there are, including the Euler vector field,
$\frac{p}{2}+1$ linearly independent vector fields in $Lift(f)$ with
linear terms in some component.
\end{lem}
\begin{proof}
First of all, $n$ is odd since there is no codimension 1 primitive
germs when $n$ is even (\cite{robertamond}). The normal form for
$f:(\C^{2k-3},0)\rightarrow (\C^{2k-2},0)$ is
$$f(u_1,\ldots,u_{k-2},v_1,\ldots,v_{k-2},x)=(\underline{u},\underline{v},x^{k}+\sum_{i=1}^{k-2}u_ix^i,x^{k+1}+\sum_{i=1}^{k-2}v_ix^i).$$
Consider the 1 parameter versal unfolding of $f$ as
$$F(u_1,\ldots,u_{k-2},v_1,\ldots,v_{k-1},x)=(\underline{u},\underline{v},x^{k}+\sum_{i=1}^{k-2}u_ix^i,x^{k+1}+\sum_{i=1}^{k-1}v_ix^i),$$
and consider the variables in the target as
$(U_1,\ldots,U_{k-2},V_1,\ldots,V_{k-1},W_1,W_2)$. We proceed as in
\cite{nishilift} to obtain the linear part of vector fields in
$Lift(F)$, that is, we study the equation $dF\circ\xi=\eta\circ F$
modulo $F^*(\mathcal M_{2k-1}^2)$. Such vector fields for a slightly
different normal form for $F$ have been studied in
\cite{houstonlittle}. We obtain that the linear part of vector
fields in $Lift(F)$ are generated by the following $3k-2$ elements:
$$\{\sum_{i=1}^{k-2}(k-i)U_i\frac{\partial}{\partial
U_i}+\sum_{i=1}^{k-1}(k-i+1)V_i\frac{\partial}{\partial
V_i}+kW_1\frac{\partial}{\partial
W_1}+(k+1)W_2\frac{\partial}{\partial W_2},$$
$$\sum_{i=1}^{k-2}(V_i-\frac{i-1}{k}U_{i-1})\frac{\partial}{\partial
U_i}+\sum_{i=1}^{k-2}\frac{k-i+1}{k}V_i\frac{\partial}{\partial
V_{i+1}}-\frac{k+1}{k}W_2\frac{\partial}{\partial
V_1}+W_2\frac{\partial}{\partial W_1},$$
$$\sum_{i=1}^{k-1}(V_i-U_{i-1})\frac{\partial}{\partial
V_i}+W_1\frac{\partial}{\partial V_1},
\sum_{i=1}^{k-2}(V_i-U_{i-1})\frac{\partial}{\partial
U_i}+W_1\frac{\partial}{\partial U_1}+W_2\frac{\partial}{\partial
W_1},$$ $$\sum_{i=1}^{k-j-1}(V_i-U_{i-1})\frac{\partial}{\partial
U_{i+j-1}}+W_2\frac{\partial}{\partial
U_{j-1}}+W_1\frac{\partial}{\partial U_j}, 2\leq j\leq k-2,
W_2\frac{\partial}{\partial U_{k-2}},$$
$$\sum_{i=1}^{k-j-2}((\frac{i-2}{k}U_{i-2}-V_i)\frac{\partial}{\partial
U_{i+j}}+(\frac{i}{k}V_i-\frac{k+1}{k}U_{i-1})\frac{\partial}{\partial
V_{i+j+1}})+W_2\frac{\partial}{\partial
U_{j}}+\frac{k+1}{k}W_1\frac{\partial}{\partial V_{j+2}},$$ $$1\leq
j \leq k-3, W_2\frac{\partial}{\partial V_{k-1}},
\sum_{i=1}^{k-j-1}(U_{i-1}-V_i)\frac{\partial}{\partial
V_{i+j}}+W_2\frac{\partial}{\partial
V_{j}}-W_1\frac{\partial}{\partial V_{j+1}},1\leq j\leq k-2\},$$
where $U_0=U_{-1}=0$.

The section $V_{k-1}=0$ of the discriminant of $F$ is the
discriminant of $f$. We proceed as in the above Lemma taking the
corresponding $2k-2$ components of vector fields in $Lift(F)$ which
are tangent to $V_{k-1}=0$ and substitute $V_{k-1}$ by 0. The linear
part of the generators of vector fields in $Lift(f)$ are the first,
the fourth, the fifth (which comprises $k-3$ vector fields) and the
sixth vector fields in the list above, which means that there are
exactly $k$ linearly independent vector fields with linear part in
$Lift(f)$.
\end{proof}

\begin{prop}\label{codprimmorse}
If $h=\{f,g\}$ is a multigerm with $f$ a primitive monogerm of
$\mathcal A_e$-codimension 1 and $g$ a prism on a Morse function or
an immersion, then $h$ has codimension greater than or equal to $p$
when $n\geq p$ and greater than or equal to $\frac{p}{2}$ when
$n=p-1$.
\end{prop}
\begin{proof}
Considering the normal forms of the primitive $\mathcal
A_e$-codimension 1 monogerms given in Proposition \ref{a2an} and
Lemma \ref{liftprim} the least codimension of $h$ occurs when $g$ is
such that $g\pitchfork f$ and $g$ is transversal to all the limits
of tangent spaces to the strata of the stratification by stable
types in the target, so we take
$g(x_1,\ldots,x_n)=(x_1,\ldots,x_{p-1},\sum_{i=p}^n x_i^2+x_{p-1})$.
Consider the short exact sequence:

\begin{equation*}
\begin{CD}
0@>>> \frac{\theta(g)}{tg(\theta_{n})+wg(Lift(f))} @>>> N\mathcal
A_e(\{f,g\}) @>{\overline{h}}>> N\mathcal A_e(f) @>>> 0
\end{CD}
\end{equation*}

From here we have that
$$\mathcal{A}_e-cod(\{f,g\})=\mathcal{A}_e-cod(f)+\dim_{\mathbb C}
\frac{\theta(g)}{tg(\theta_{n})+wg(Lift(f))}.$$ Now,
$\frac{\theta(g)}{tg(\theta_{n})+wg(Lift(f))}$ is isomorphic to
$$\frac{\mathcal O_{n}\oplus\mathcal
O_n}{\{w:w=(v_{p-1}(x)+\eta_{p-1}\circ g(x),v_{p-1}(x)+\sum_{i=p}^n
2x_iv_i(x)+\eta_p\circ g(x))\}}$$ by projection to the last two
components, where $v_i\in \mathcal O_n$ and $\eta_{p-1}$ and
$\eta_{p}$ are the last two components of the vector fields in
$Lift(f)$. And this is isomorphic to $$\frac{\mathcal
O_{p-1}}{\{w_2:w_2=-\eta_{p-1}(x_1,\ldots,x_{p-1},x_{p-1})+\eta_{p}(x_1,\ldots,x_{p-1},x_{p-1})\}}.$$
When $n\geq p$, from Lemma \ref{primngeqp}, the components of vector
fields in $Lift(f)$ other than the Euler vector field are quadratic
or of higher order, so the dimension of this space is greater than
or equal to $p-1$. Finally we have that the codimension of $h$ is
greater than or equal to $p$.

When $n=p-1$, from the above Lemma we have that there are at most
$\frac{p}{2}+1$ linearly independent vector fields in $Lift(f)$ with
linear parts, so the dimension of the quotient is greater than or
equal to $p-(\frac{p}{2}+1)=\frac{p}{2}-1$. Finally we have that the
codimension of $h$ is greater than or equal to $\frac{p}{2}$.
\end{proof}

\begin{rem}\label{ejemcati}
\begin{enumerate}
\item[1)] When $p=1$, the bigerm formed by a Morse function and an $A_2$-singularity has codimension 2.
\item[2)]
If $n=1$ and $p=2$, we have the codimension 2 bigerm
$\{(x^2,x^3),(0,x)\}$, and if $n=p=2$ we have the codimension 2
bigerm
\begin{equation}
\begin{cases}
(x^4+yx,y)\\
(x,y^2+x)
\end{cases}
\end{equation}

\item[3)]
In the equidimensional case, given the bigerm
\begin{equation}
\begin{cases}
(x_1^{n+2}+x_2x_1+\ldots+x_nx_1^{n-1},x_2,\ldots,x_n)\\
(x_1,\ldots,x_{n-1},x_{n}^2+x_{n-1})
\end{cases}
\end{equation}
the codimension is exactly $n$ (except when $n=1$, see case 1)
above) and is non-simple when $n>2$.

\item[4)] When $(n,p)=(3,4)$, the bigerm
\begin{equation}
\begin{cases}
(u,v,x^3+ux,x^4+vx)\\
(u,v,u,x)
\end{cases}
\end{equation}
has codimension 2.
\end{enumerate}
\end{rem}

\begin{rem}\label{primmultmorse}
Proposition \ref{codprimmorse} can be adapted to the case where $f$
is a primitive codimension 1 multigerm. For example, in the case
$n=p-1$, suppose that $f=\{f_0,g_0\}$ is a codimension 1 monic
concatenation with $\widetilde\tau(f_0)=\{0\}$ ($f_0$ is stable and
$n=2k-2$ must be even). That is
$$f_0(x_1,\ldots,x_n)=(x_1,\ldots,x_{n-1},x_n^{\frac{n}{2}+1}+\sum_{i=1}^{\frac{n}{2}-1}x_ix_n^i,x_n^{\frac{n}{2}+2}+\sum_{i=\frac{n}{2}}^{n-1}x_ix_n^i)$$
and $g_0(x_1,\ldots,x_n)=(x_1,\ldots,x_{n-2},0,x_{n-1},x_{n})$.
$Lift(f_0)$ is as in Lemma \ref{liftprim} and $Lift(g_0)=\langle
\frac{\partial}{\partial U_1},\ldots,V_{k-1}\frac{\partial}{\partial
V_{k-1}},\frac{\partial}{\partial W_1},\frac{\partial}{\partial
W_2}\rangle$, so $Lift(f)=Lift(\{f_0,g_0\})$ is the intersection,
and there are only $k$ vector fields with linear parts. Analogously
to the proof of Proposition \ref{codprimmorse}, in order to get the
least codimension of $h$, we take
$g(x_1,\ldots,x_n)=(x_1,\ldots,x_{n-2},x_{n},x_{n-1},x_{n}).$ We
must calculate the dimension of
$$\frac{\mathcal
O_{p-1}}{\{w_2:w_2=-\eta_{p-2}(x_1,\ldots,x_{p-3},x_{p-1},x_{p-2},x_{p-1})+\eta_{p}(x_1,\ldots,x_{p-3},x_{p-1},x_{p-2},x_{p-1})\}}.$$
Out of the vector fields with non zero linear parts in $Lift(f)$,
only the first (as shown in Lemma \ref{liftprim}) has components in
$\frac{\partial}{\partial V_{k-1}}$ and $\frac{\partial}{\partial
W_2}$. Therefore, the dimension of the quotient is at least $p-1$
since we are missing constants and the linear terms
$x_1,\ldots,x_{p-2}$. This means that $\mathcal A_e-\cod(h)\geq
\mathcal A_e-\cod(f)+p-1=p$.

If $f_0$ is a stable multigerm with $\widetilde\tau(f_0)=\{0\}$, we
can repeat the diagonal block matrix argument from the proof of
Theorem \ref{codn} to prove that $\mathcal A_e-\cod(h)\geq p$.
\end{rem}

\begin{coro}\label{corocod1}
Let $p>2$ and $h=\{f,g\}$ be multigerm of $\mathcal A_e$-codimension
2 with $f$ a monogerm of codimension 1, then $g$ is a prism on a
Morse function or an immersion and $h$ is an augmentation and
concatenation except for the normal form from Remark \ref{ejemcati}
4), where we have a primitive monogerm and an immersion of
codimension 2. Furthermore, if the multiplicity of $f$ is less than
or equal to $p$ in the case $n\geq p$ or to $[\frac{n}{2}]$ in the
case $n=p-1$, then it is also an augmentation.
\end{coro}

\begin{prop}\cite[Propostion 5.16]{robertamond}\label{robmond}
Let $h=\{f,g\}$ be a primitive $\mathcal A_e$-codimension 1
multigerm, and suppose that $g$ is not transverse to
$\widetilde\tau(f)$. Then
\begin{enumerate}
\item[i)] if moreover $g$ and $f$ are transverse, it follows that
$g$ is a prism on a Morse function or an immersion and $h$ is a
monic concatenation (in particular $f\pitchfork \widetilde\tau(g)$).
\item[ii)] if $g$ and $f$ are not transverse, then $p=1$, and $f$
and $g$ are both Morse functions.
\end{enumerate}
\end{prop}

From here on we suppose that $h=\{f,g\}$ is a primitive multigerm of
$\mathcal A_e$-codimension 2 with $f$ and $g$ stable and $p>1$.

\begin{lem}\label{decomp}
Let $h=\{f,g\}$ be a primitive multigerm. If it admits a 1-parameter
stable unfolding then there is a decomposition
$T_0\C^p=\widetilde\tau(f)\oplus\widetilde\tau(g)\oplus \C\{v\}$ for
some $v\in\C^p$. If not, there exist $v_1$ and $v_2$ in $\C^p$ and a
decomposition $T_0\C^p=(\widetilde\tau(f)+\widetilde\tau(g))\oplus
\C\{v_1,v_2\}$ where all the sums are direct if and only if
$\dim(\widetilde\tau(f)\cap\widetilde\tau(g))=0$.
\end{lem}
\begin{proof}
By Proposition \ref{almreg2}, if $h$ admits a 1-parameter stable
unfolding, then $\widetilde\tau(f)$ and $\widetilde\tau(g)$ have
almost regular intersection. Therefore
$\cod(\widetilde\tau(f)+\widetilde\tau(g))=\cod\widetilde\tau(f)+\cod\widetilde\tau(g)-\cod(\widetilde\tau(f)\cap\widetilde\tau(g))=1$.
So we have
$T_0\C^p=(\widetilde\tau(f)+\widetilde\tau(g))\oplus\C\{v\}$. On the
other hand, since $h$ is primitive, by Corollary \ref{aug}, we have
that $\cod\widetilde\tau(f)+\cod\widetilde\tau(g)>p$, therefore
$\cod(\widetilde\tau(f)\cap\widetilde\tau(g))>p-1$ and
$\dim(\widetilde\tau(f)\cap\widetilde\tau(g))=0$, which proves the
result for the first case.

In the case that there is no stable 1-parameter unfolding, by
Proposition \ref{almreg} the analytic strata have almost regular
intersection of order 2.
\end{proof}

\begin{lem}\label{casonoopsu}
Suppose that $p>2$, then there is no $h=\{f,g\}$ such that
$f\pitchfork g$, $g$ is not transverse to $\widetilde\tau(f)$ and
$f$ is not transverse to $\widetilde\tau(g)$.
\end{lem}
\begin{proof}
Suppose there is such an $h$. If $f$ is a prism on a Morse function
or an immersion, then $Im(df_0)=\widetilde\tau(f)$, and so
$f\pitchfork g$ implies $f\pitchfork \widetilde\tau(g)$, which is a
contradiction. Therefore $f$ is not a prism on a Morse function or
an immersion. Equally for $g$.

Now suppose that $\widetilde\tau(f)=\{0\}$. By Proposition
\ref{almreg2}, $2\geq
\cod\widetilde\tau(f)+\cod\widetilde\tau(g)-\cod(\widetilde\tau(f)\cap\widetilde\tau(g))=\cod\widetilde\tau(g).$
Since $g$ is not a prism on a Morse function or an immersion,
$\cod\widetilde\tau(g)=2$. In the case $n\geq p$, Theorems
\ref{codn} and \ref{codndoublefold} imply that $\mathcal
A_e-\cod(h)\geq n\geq p>2$. This contradicts that $\mathcal
A_e-\cod(h)=2$ and so $\widetilde\tau(f)\neq\{0\}$ (the same is
valid for $g$). In the case $p=n+1$, since there are no stable
monogerms with codimension 2 analytic stratum, $g$ must be a double
immersion. Remark \ref{primmultmorse} implies that $\mathcal
A_e-\cod(h)\geq p>2$, which again is a contradiction.


Therefore $\widetilde\tau(f)\neq\{0\}\neq\widetilde\tau(g)$. Since
neither $f$ or $g$ can be a prism on a Morse function or an
immersion, we have that $1<\cod\widetilde\tau(f)<p$ and
$1<\cod\widetilde\tau(g)<p$. We construct a 1-parameter deformation
$h_{u}$ of $h$ by constructing a 1-parameter generic deformation of
$g$ such that $\widetilde\tau(g)$ remains fixed and $g_u$ becomes
transverse to $\widetilde\tau(f)$. Since $f$ is still not transverse
to $\widetilde\tau(g)$, $\mathcal A_e-\cod(h_u)\geq 1$. If $\mathcal
A_e-\cod(h_u)=1$, we construct a 1-parameter generic deformation of
$f$ such that $\widetilde\tau(f)$ remains fixed and $f_v$ becomes
transverse to $\widetilde\tau(g)$. It follows that $h_{(u,v)}$ is
not equivalent to $h_u=h_{(u,0)}$. Since $\widetilde\tau(f)$ and
$\widetilde\tau(g)$ remain fixed, they are still not transverse and
so $h_{(u,v)}$ is not stable for $v\neq 0$. This would imply that
for each $u$ we have a 1-parameter deformation where each member is
not equivalent to $h_u$ and is not stable. This is impossible since
$h_u$ has codimension 1, and by \cite{robertamond} all codimension 1
germs are simple. So $\mathcal A_e-\cod(h_u)=2$ and $h_u$ is not
equivalent to $h$ for $u\neq 0$. Therefore, either $h$ is non-simple
or $\mathcal A_e-\cod(h)>2$.
\end{proof}

\begin{prop}\label{nottrans}
Suppose that $g=\{g_1,\ldots,g_r\}$ is not transverse to
$\widetilde\tau(f)$ and $f\pitchfork g$, then
\begin{enumerate}
\item[i)] If $r=1$ and $Im(dg_0)=\widetilde\tau(g)$, then $h$ is a monic concatenation.

\item[ii)] If $r=1$ and $Im(dg_0)\supsetneqq\widetilde\tau(g)$, then $h$ is one of the following

\begin{itemize}
\item a (non-monic) generalised concatenation with $g$,
\item $g$ is an $A_2$-singularity and $f$ is either an $A_2$-singularity or a bigerm of two prisms on Morse functions (only if $n\geq p=2$).
\end{itemize}

\item[iii)] If $r>1$, then $r=2$ and $h$ is one of the following

\begin{itemize}
\item a double fold (immersion) concatenation with $g$,
\item a trigerm of an $A_2$-singularity with two prisms on Morse functions (only if $n\geq p=2$).
\end{itemize}
\end{enumerate}

\end{prop}
\begin{proof}
ii) Suppose $Im(dg_0)\supsetneqq\widetilde\tau(g)$. If
$\widetilde\tau(f)=\widetilde\tau(g)=\{0\}$, then
$\cod\widetilde\tau(f)+\cod\widetilde\tau(g)=2p$ and
$\cod(\widetilde\tau(f)\cap\widetilde\tau(g))=p$, therefore, by
Proposition \ref{almreg2}, $p\leq 2$. From the known classifications
mentioned at the beginning of the section, the only possibilities
are that $n\geq p=2$ and $h$ is a bigerm with two
$A_2$-singularities or a trigerm where $g$ is an $A_2$-singularity
and $f$ is a bigerm of two prisms on Morse functions. If $p>2$ we
can assume that $\widetilde\tau(g)\neq 0$. In fact, if
$\widetilde\tau(g)=\{0\}$, then $f$ is not transversal to
$\widetilde\tau(g)$. Since by hypothesis $g$ is not transversal to
$\widetilde\tau(f)$, it follows by Lemma \ref{casonoopsu} that there
is no such $h$.

Therefore, $1\leq p-s=\dim\widetilde\tau(g)<p-1$ and we are in the
case of a (non-monic) generalised concatenation (and $p>2$): if $h$
admits a 1-parameter stable unfolding we have a partition
$T_0\C^p=\C^{p-s}\times\C^{s-1}\times\C$, where
$\C^{p-s}\times\{0\}\times\{0\}$ is the analytic stratum of $g$ and
$\{0\}\times\C^{s-1}\times\{0\}$ is the analytic stratum of $f$. By
adequate changes of coordinates in the source we can take $g$ to the
form
$$g(x_1,\ldots,x_n)=(x_1,\ldots,x_{p-s},g_{p-s+1}(x_{p-s+1},\ldots,x_n),\ldots,g_{p}(x_{p-s+1},\ldots,x_n)).$$

Now, by a change of variables in the source, we can write
$f(x_1,\ldots,x_n)=(f_1(x_1,\ldots,x_n),\ldots,f_{p-s}(x_1,\ldots,x_n),x_{n-s+1},\ldots,x_{n-1},f_p(x_1,\ldots,x_n)).$
Since $f\pitchfork g$, if $\{0\}\times\{0\}\times \C\nsubseteq
Im(df_0)$, then $\{0\}\times\{0\}\times \C\subseteq Im(dg_0)$ and so
$g\pitchfork \widetilde\tau(f)$, which contradicts the hypothesis.
Therefore, we can take $f$ to the form
$$f(x_1,\ldots,x_n)=(f_1(x_1,\ldots,x_n),\ldots,f_{p-s}(x_1,\ldots,x_n),x_{n-s+1},\ldots,x_n),$$
so $f$ is an $s$-parameter stable unfolding of a certain $f_0$.

Now suppose that $h$ does not admit a 1-parameter stable unfolding.
If $\widetilde\tau(f)\cap\widetilde\tau(g)=\{0\}$ we have a
partition $\C^{p-s}\times\C^{s-2}\times\C^2$, where
$\widetilde\tau(g)=\C^{p-s}\times\{0\}\times\{0\}$ and
$\widetilde\tau(f)=\{0\}\times\C^{s-2}\times\{0\}$. Similarly as
above, we write
$$g(x_1,\ldots,x_n)=(x_1,\ldots,x_{p-s},g_{p-s+1}(x_{p-s+1},\ldots,x_n),\ldots,g_{p}(x_{p-s+1},\ldots,x_n))$$
and $f$ as
$$(f_1(x_1,\ldots,x_n),\ldots,f_{p-s}(x_1,\ldots,x_n),x_{n-s+1},\ldots,x_{n-2},f_{p-1}(x_1,\ldots,x_n),f_p(x_1,\ldots,x_n)).$$
Since $f\pitchfork g$ and $g$ is not transverse to
$\widetilde\tau(f)$, then $\{0\}\times\{0\}\times
\C\times\C\nsubseteq Im(dg_0)$. If $\{0\}\times\{0\}\times
\C\times\C\subseteq Im(df_0)$ we can take $f$ to the desired form.
If $Im(dg_0)=\C^{p-s}\times\{0\}\times\C\times\{0\}$ and
$Im(df_0)=\{0\}\times\C^{s-2}\times\{0\}\times\C$ then $f$ is not
transversal to the analytic stratum of $g$. It follows by Lemma
\ref{casonoopsu} that there is no such $h$ in codimension 2.

If $\dim\widetilde\tau(f)\cap\widetilde\tau(g)=k>0$, the only
difference with the above case is that the analytic stratum of $f$
overlaps the analytic stratum of $g$ in $k$ directions. So now
$$\begin{cases}
(x_1,\ldots,x_{p-s},g_{p-s+1}(x_{p-s+1},\ldots,x_n),\ldots,g_{p}(x_{p-s+1},\ldots,x_n))\\
(f_1(x_1,\ldots,x_n),\ldots,f_{p-s-k}(x_1,\ldots,x_n),x_{n-s-k+1},\ldots,x_{n-2},f_{p-1}(x_1,\ldots,x_n),f_p(x_1,\ldots,x_n)).
\end{cases}$$
We proceed analogously.

i) If $Im(dg_0)=\widetilde\tau(g)$ then $g$ is either a prism on a
Morse function or an immersion and so $\cod(\widetilde\tau(g))=1$.
If $h$ admits a 1-parameter stable unfolding, since $h$ is primitive
by Corollary \ref{aug}
$\cod\widetilde\tau(f)+\cod\widetilde\tau(g)>p$ and so
$\widetilde\tau(f)=\{0\}$. Now proceed as in the proof of
\cite[Proposition 5.16]{robertamond} with the only difference that
$f$ is a 1-parameter stable unfolding of an $f_0$ of codimension 2
(the procedure is similar to the one above).

If $h$ does not admit a 1-parameter stable unfolding, by Proposition
\ref{almreg2} we have that $\cod \widetilde\tau(f)+1-\cod
(\widetilde\tau(f)\cap\widetilde\tau(g))=2$ and so $\cod
(\widetilde\tau(f)\cap\widetilde\tau(g))=\cod \widetilde\tau(f)-1$,
which is impossible.

iii) First suppose that $h_i=\{f,g_i\}$ is stable for all $i$. Then
$\widetilde\tau(g_i)\pitchfork\widetilde\tau(f)$ and therefore
$g_i\pitchfork\widetilde\tau(f)$ and so
$g\pitchfork\widetilde\tau(f)$, which contradicts the hypothesis. So
there exists $i_0\in\{1,\ldots,r\}$ such that
$h_{i_0}=\{f,g_{i_0}\}$ has $\mathcal A_e$-codimension 1. In this
case, by Lemma \ref{cod2}, $g=\{g_{i_0},g_{i_1}\}$, where $g_{i_1}$
is a prism on a Morse function or an immersion, and so $r=2$.

Suppose that $g_{i_0}$ is not transverse to $\widetilde\tau(f)$. If
$g_{i_0}$ is not transverse to $f$, by Proposition \ref{robmond},
$p=1$, which is a contradiction. If $g_{i_0}\pitchfork f$, again by
Proposition \ref{robmond}, $h_{i_0}$ is a monic concatenation where
$g_{i_0}$ is a prism on a Morse function or an immersion.

Now suppose that $g_{i_0}\pitchfork\widetilde\tau(f)$. Since $g$ is
not transverse to $\widetilde\tau(f)$, then $g_{i_1}$ cannot be
transverse to $\widetilde\tau(f)$, therefore $h_{i_1}=\{f,g_{i_1}\}$
has codimension 1 and $g_{i_0}$ is a prism on a Morse function or an
immersion. Since $g$ is stable, $g_{i_0}\pitchfork g_{i_1}$ and
$\dim\widetilde\tau(g)=p-2$.

If $p=2$, $\cod\widetilde\tau(f)\leq 2$. If it is equal to 2 then
$f$ can be either an $A_2$-singularity or two prisms on Morse
functions. However, the second case does not occur since $h$ has
codimension 2. If it is equal to 1, then $f$ is a prism on a Morse
function or an immersion and $g\pitchfork f$ implies $g\pitchfork
\widetilde\tau(f)$, which contradicts the hypothesis.

If $p>2$ we take $g=\{g_{i_0},g_{i_1}\}$ to the form
\begin{equation}
\begin{cases}
(x_1,\ldots,x_{p-2},x_{p-1},\sum_{i=p}^{n}x_i^2)\\
(x_1,\ldots,x_{p-2},x_{p-1},\sum_{i=p}^{n}x_i^2+x_{p-1})
\end{cases}
\end{equation} and proceed as
in case ii).
\end{proof}

\begin{prop}\label{nadatrans}
If $g$ and $f$ are not transverse then $h$ is one of the following

\begin{itemize}
\item an augmentation and concatenation,
\item $f$ is a Morse function and $g$ is an $A_2$-singularity (only if $n\geq p=2$),
\item one of the following normal forms (when $p=n+1$ and $n$ is even):
\begin{equation}
n=2
\begin{cases}
(x,y^2,xy)\\
(x,x^2,y)
\end{cases}
,\,\,\,n=4
\begin{cases}
(u_1,v_1,v_2,y^3+u_1y,v_1y+v_2y^2)\\
(u_1,v_1,v_2,u_1^2+v_2,y)
\end{cases}
,
\end{equation}
\begin{equation}
n=2k-2, k\geq 4
\begin{cases}
(u_1,\ldots,u_{k-2},v_1,\ldots,v_{k-1},y^k+\sum_{i=1}^{k-2}u_iy^i,\sum_{i=1}^{k-1}v_iy^i)\\
(u_1,\ldots,u_{k-2},v_1,\ldots,v_{k-2},u_{k-3}+u_{k-2}^2,v_{k-1},y)
\end{cases}
\end{equation}
\end{itemize}
\end{prop}
\begin{proof}
If $f$ is not transverse to $g$, then $f$ is not transverse to
$\widetilde\tau(g)$ and $g$ is not transverse to
$\widetilde\tau(f)$. If $p=2$ and $h$ admits a 1-parameter stable
unfolding, since $h$ is primitive, by Corollary \ref{aug}
$\cod\widetilde\tau(f)+\cod\widetilde\tau(g)>2$ and since the
analytic stratum must have codimension less than or equal to $p=2$,
$\cod\widetilde\tau(g)=1$ and $\cod\widetilde\tau(f)=2$ (or
viceversa), so $h$ is a non-transversal bigerm of a prism on a Morse
function and an $A_2$-singularity or a trigerm of three prisms on
Morse functions where two of them are non-transversal and the third
is transversal to the other two (the latter is an augmentation and
concatenation). If $h$ does not admit a 1-parameter stable
unfolding, then by Proposition \ref{almreg2}
$\cod\widetilde\tau(f)=\cod\widetilde\tau(g)=2$ and, by the known
classifications mentioned at the beginning of this section, there
are no possibilities.

If $p>2$, let $f=\{f_1,\ldots,f_r\}$, $r>1$, as $f$ is not
transverse to $g$, there exists $f_{i_0}$ which is not transverse to
$g$, therefore $\mathcal A_e-\cod(\{f_{i_0},g\})=1$. By Proposition
\ref{cod2}, $f=\{f_{i_0},f_{i_1}\}$ with $f_{i_1}$ a prism on a
Morse function or an immersion and $f_{i_0}\pitchfork f_{i_1}$. We
have that $g$ is not transverse to $\widetilde\tau(f_{i_0})$ and $g$
is not transverse to $f_{i_0}$. If $\{f_{i_0},g\}$ is primitive,
then by Proposition \ref{robmond} $p=1$ and we have a contradiction,
therefore $\{f_{i_0},g\}$ is an augmentation and $h$ is an
augmentation and concatenation which is not an augmentation.

Now suppose that $f$ and $g$ are monogerms (if $g$ is not a
monogerm, we change it for $f$ and proceed as above). We have that
$\widetilde\tau(f)=\{0\}$ if and only if $\cod\widetilde\tau(g)=1$,
and the same changing $f$ for $g$. In fact, suppose
$\widetilde\tau(f)=\{0\}$ ($\cod\widetilde\tau(f)=p$), then by
Proposition \ref{almreg2}, $\cod\widetilde\tau(g)\leq 2$. We have
that $\cod\widetilde\tau(g)\neq 0$ because $g$ is not a submersive
branch. In the case $p=n+1$, from the known classifications, there
is no stable monogerm with $\cod\widetilde\tau(g)=2$ and if $n\geq
p$, by Theorems \ref{codn} and \ref{codndoublefold}, if
$\cod\widetilde\tau(g)=2$, $\mathcal A_e-\cod(h)\geq n\geq p>2$, so
$\cod\widetilde\tau(g)=1$. On the other hand, suppose that
$\cod\widetilde\tau(g)=1$. If $h$ admits a 1-parameter stable
unfolding, by Corollary \ref{aug} $\widetilde\tau(f)=\{0\}$. If $h$
does not admit a 1-parameter stable unfolding, by Proposition
\ref{almreg2} we have that $\cod
(\widetilde\tau(f)\cap\widetilde\tau(g))=\cod \widetilde\tau(f)-1$,
which is impossible.

Suppose $\widetilde\tau(f)=\{0\}$ and $Im(dg_0)=\widetilde\tau(g)$.
In the case $n\geq p$, $f$ is an $A_n$-singularity. Using the exact
sequence in the proofs of Theorems \ref{cuspconc} and
\ref{doubleconc} and the information about $Lift(A_n)$ in
Proposition \ref{a2an} we can see that the codimension in this case
is greater than or equal to $n$. In the case $p=n+1$, from
\cite{houstonroberta} and \cite{roberta} we have that the only
possibilities are when $n$ is even. The bigerm has a normal form
$\{(x,y^2,xy),(x,x^2,y)\}$ for the case $n=2$. For other $n$, the
normal forms can be found in \cite{houstonroberta}. If
$f:(\C^{2k-2},0)\rightarrow (\C^{2k-1},0)$, then
$f(u_1,\ldots,u_{k-2},v_1,\ldots,v_{k-1},y)=(u_1,\ldots,u_{k-2},v_1,\ldots,v_{k-1},y^k+\sum_{i=1}^{k-2}u_iy^i,\sum_{i=1}^{k-1}v_iy^i)$.
For $k=3$, $g(u_1,v_1,v_2,y)=(u_1,v_1,v_2,u_1^2+v_2,y)$. For $k\geq
4$ the normal form for the bigerm is:

\begin{equation}
\begin{cases}
(u_1,\ldots,u_{k-2},v_1,\ldots,v_{k-1},y^k+\sum_{i=1}^{k-2}u_iy^i,\sum_{i=1}^{k-1}v_iy^i)\\
(u_1,\ldots,u_{k-2},v_1,\ldots,v_{k-2},u_{k-3}+u_{k-2}^2,v_{k-1},y)
\end{cases}
\end{equation}

Now suppose that $1<\cod\widetilde\tau(f)<p$ and
$1<\cod\widetilde\tau(g)<p$. In this case we proceed exactly as in
the proof of Lemma \ref{casonoopsu} and obtain that either $h$ is
non-simple or $\mathcal A_e-\cod(h)>2$.
\end{proof}

\begin{prop}\label{trans}
If $f\pitchfork\widetilde\tau(g)$ and $g\pitchfork\widetilde\tau(f)$
then $h=\{f,g\}$ is a non-monic generalised concatenation.
\end{prop}
\begin{proof}
The fact that $f\pitchfork\widetilde\tau(g)$ and
$g\pitchfork\widetilde\tau(f)$ implies that
$\widetilde\tau(g)\neq\{0\}\neq\widetilde\tau(f)$ so, if $h$ admits
a 1-parameter stable unfolding, again we have a decomposition of
$\C^p$ as $\C^{p-s}\times\C^{s-1}\times\C$ where $s>1$,
$\widetilde\tau(g)=\C^{p-s}\times\{0\}\times\{0\}$ and
$\widetilde\tau(f)=\{0\}\times\C^{s-1}\times\{0\}$. Let
$z_1,\ldots,z_p$ be the coordinates of $\C^p$. Since
$f\pitchfork\widetilde\tau(g)$, we can take $z_p\circ f$ as a
coordinate, $u$, on the domain of $f$ and since
$g\pitchfork\widetilde\tau(f)$, we can take $z_p\circ g$ as a
coordinate $u$ on the domain of $g$. A coordinate change now takes
$h=\{f,g\}$ to the form

\begin{equation}
\begin{cases}
(x,Y,u)\mapsto (f_{u,Y}(x),Y,u)\\
(X,y,u)\mapsto (X,g_u(y),u)
\end{cases}
\end{equation}
which is clearly a generalised concatenation.

If $h$ does not admit a 1-parameter stable unfolding $h$ can be
taken to the form
\begin{equation}
\begin{cases}
(x,Y,u)\mapsto (f_{u,v,Y}(x),Y,u,v)\\
(X,y,u)\mapsto (X,g_{u,v}(y),u,v)
\end{cases}
\end{equation}
which is a generalised concatenation too.
\end{proof}

In summary we have:

\begin{teo}\label{teofinal}
Let $h=\{f,g\}$ be of $\mathcal A_e$-codimension 2, then
\begin{enumerate}
\item[1)] if $f$ is a monogerm of $\mathcal A_e$-codimension 1, then
$g$ a prism on a Morse function or an immersion and
\begin{enumerate}
\item[i)] $h$ is an augmentation if and only if $f$ is an augmentation with $m_0(f)\leq p$ when $n\geq p$ ($m_0(f)\leq [\frac{n}{2}]$ when $p=n+1$),

\item[ii)] $h$ is an augmentation and concatenation if $f$ is an augmentation with $m_0(f)=p+1$ when $n\geq p$ ($m_0(f)=[\frac{n}{2}]+1$ when $p=n+1$),

\item[iii)] if $p=1,2$ and $m_0(f)=p+2$ when $n\geq p$ ($m_0(f)=[\frac{n}{2}]+2$ when $p=n+1$) then $f$ is a primitive monogerm of codimension 1,

\item[iv)] if $(n,p)=(3,4)$ and $m_0(f)=3$ then $h$ has the normal form in Remark \ref{ejemcati} 4),
\end{enumerate}
\item[2)] if $f$ and $g$ are stable, then
\begin{enumerate}
\item[i)] $\cod(\widetilde\tau(f))+\cod(\widetilde\tau(g))\leq p$ if and only if $h$ is an augmentation,

\item[ii)] if $h$ is primitive and $g$ is not transverse to $\widetilde\tau(f)$, then
\begin{enumerate}
\item[a)] if $f\pitchfork g$, then
\begin{enumerate}
\item[a1)] Suppose $g$ is a monogerm. When $Im(dg_0)=\widetilde\tau(g)$, $h$ is a monic concatenation. When $Im(dg_0)\supsetneqq\widetilde\tau(g)$, then either $h$ is a (non-monic) generalised concatenation with $g$, it is a bigerm with two $A_2$-singularities or it is a trigerm of an $A_2$-singularity with two prisms on Morse functions (only if $n\geq p=2$).

\item[a2)] Suppose $g$ is a multigerm, then it is a bigerm and either $h$ is a double fold (immersion) concatenation with $g$ or it is a trigerm of an $A_2$-singularity with two prisms on Morse functions (only if $n\geq p=2$).
\end{enumerate}
\item[b)] If $g$ and $f$ are not transverse then $f$ is a Morse function and $g$ is an $A_2$-singularity (only if $n\geq p=2$), $h$ is an augmentation and concatenation or it has one of the normal forms in Proposition \ref{nadatrans} (when $p=n+1$ and $n$ even).
\end{enumerate}
\item[iii)] if $h$ is primitive, $g\pitchfork \widetilde\tau(f)$ and $f\pitchfork \widetilde\tau(g)$, then $h$ is a non-monic generalised concatenation.
\end{enumerate}
\end{enumerate}
\end{teo}

\begin{rem}
If we replace $\C$ by $\R$ and analytic maps by smooth ones, all the
results in this section hold. However, in the real case, the
operations may lead to different $\mathcal A$-classes.
\end{rem}

\section{$\mathcal A_e$-codimension 2 multigerms from $\mathbb C^3$ to $\mathbb C^3$}

In this section we use the results in Section 5 in order to recover
the classification of multigerms of $\mathcal A_e$-codimension 2
from $\C^3$ to $\C^3$ obtained in \cite{mio}. First, using quadratic
and cubic augmentations ($A^2$ and $A^3$), monic concatenations
($MC$) and concatenations and augmentations (AC), we obtain all
codimension 1 and 2 germs and multigerms from $\C^2$ to $\C^2$
starting from a monogerm and the special bigerm from Proposition
\ref{robmond}, namely two Morse functions. This is shown in figure
\ref{de2en2}, where the special multigerms mentioned in Propositions
\ref{nottrans} and \ref{nadatrans} are included too.

\begin{figure}[!htb]
\centering
\includegraphics[width=0.8\linewidth]{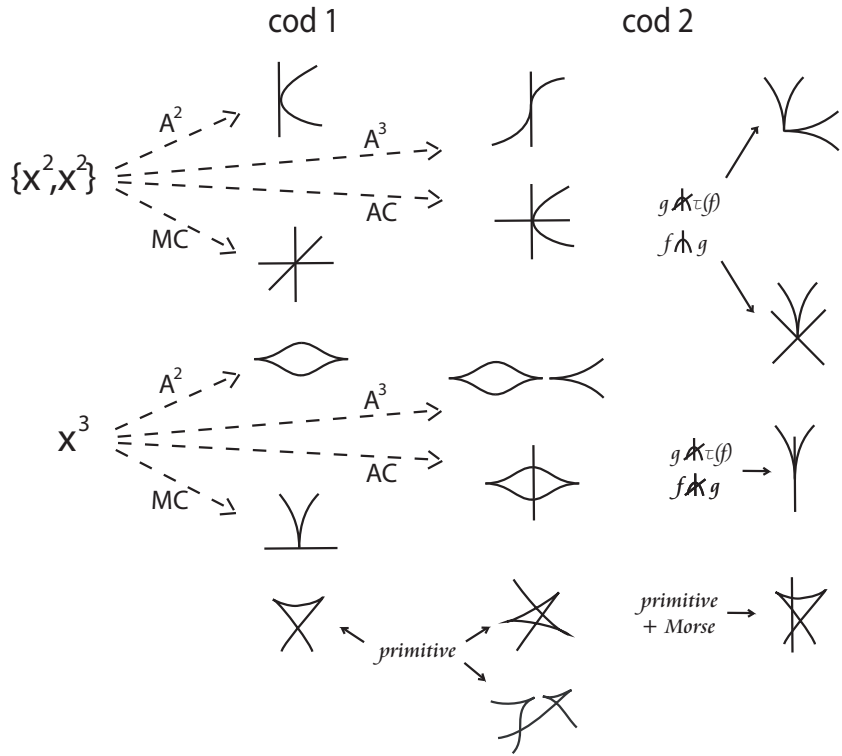}
\caption{Codimension 1 and 2 germs and multigerms of maps from
$\C^2$ to $\C^2$. The cases where a codimension 1 germ appears, a
stabilisation is represented.} \label{de2en2}
\end{figure}

The following table, obtained by W. L. Marar $\&$ F. Tari in
\cite{marartari} and earlier by V. Goryunov in \cite{Goryunov},
contains a list of normal forms for simple corank 1 monogerms of
maps from $\R^3$ to $\R^3$.

\hspace{1cm}\begin{tabular}{| c | c | c |} \hline Name & Normal form
& $\mathcal{A}_e$-codimension \\ \hline
$A_1$ & $(x,y,z^2)$ & 0 \\ \hline $3_{\mu(P)}$ & $(x,y,z^3+P(x,y)z)$ & $\mu(P)$ \\
\hline $4_1^k$ & $(x,y,z^4+xz \pm y^k z^2),k\geq1$ & $k-1$ \\
\hline $4_2^k$ &
$(x,y,z^4+(y^2 \pm x^k)z+x z^2),k\geq2$ & $k$ \\
\hline $5_1$ & $(x,y,z^5+xz+ y z^2)$ & $1$ \\ \hline $5_2$ &
$(x,y,z^5+xz+ y^2 z^2+y z^3)$ & $2$ \\
\hline
\end{tabular} \\ \\

Here $P(x,y)$ are polynomials in two variables and $\mu(P)$ denotes
the Milnor number of $P$. We add to this list the unimodular
monogerm $6_1:(x,y,z^6+yz^2+xz)$ (see \cite{Goryunov}) of $\mathcal
A_e$-codimension three.

Next we introduce the notation for germs and multigerms used in
\cite{mio}: starting from stable germs, $A_1$ (fold), $A_2$ (cusps)
and $A_3$ (swallowtails), $A_i^kA_j$ represents a multigerm with $k$
branches of type $A_i$ and a branch of type $A_j$ where  the
branches are pairwise transversal. Tangencies are indicated by $T$,
for instance, we represent by $T_{ij}$ a nondegenerate tangency
between the strata of singularities $A_i$ and $A_j$ in the branch
set, or by $T_{A_iA_j^k}$ a nondegenerate tangency between the
strata of points $A_i$ and $A_j^k$ in the discriminant, etc.
Degenerate tangencies are denoted by $DT$ . Therefore, $A_13_1$
represents a fold with a germ of type $3_1$ in the ``best" possible
position (lower contact order); $A_1T_{A_1A_1^2}$ is a quadrigerm
determined by a fold ($A_1$) with the trigerm $T_{A_1A_1^2}$, which
in turn is given by a nondegenerate tangency of a fold surface and a
double fold curve; $DT_{11}$ means a degenerate tangency between two
fold surfaces and $DT_{A_1A_1^2}$ means a degenerate tangency
between a fold surface and a double fold curve. A superindex on the
character $T$, i.e. $T^1$, denotes a special type of tangency, for
example $T^1_{22}$ means that the tangent vector to one of the
cuspidal edges is included in the tangent plane in the limit of the
other cuspidal edge; $T^1_{13}$ means that the tangent vector in the
limit of the cuspidal edges at the swallowtail point is included in
the tangent plane to the fold surface; and $T^1_{A_1^2A_2}$ means
that the tangent vector to the double point curve is included in the
tangent plane in the limit of the cuspidal edge.

\begin{figure}[!htb]
\centering
\includegraphics[width=0.99\linewidth]{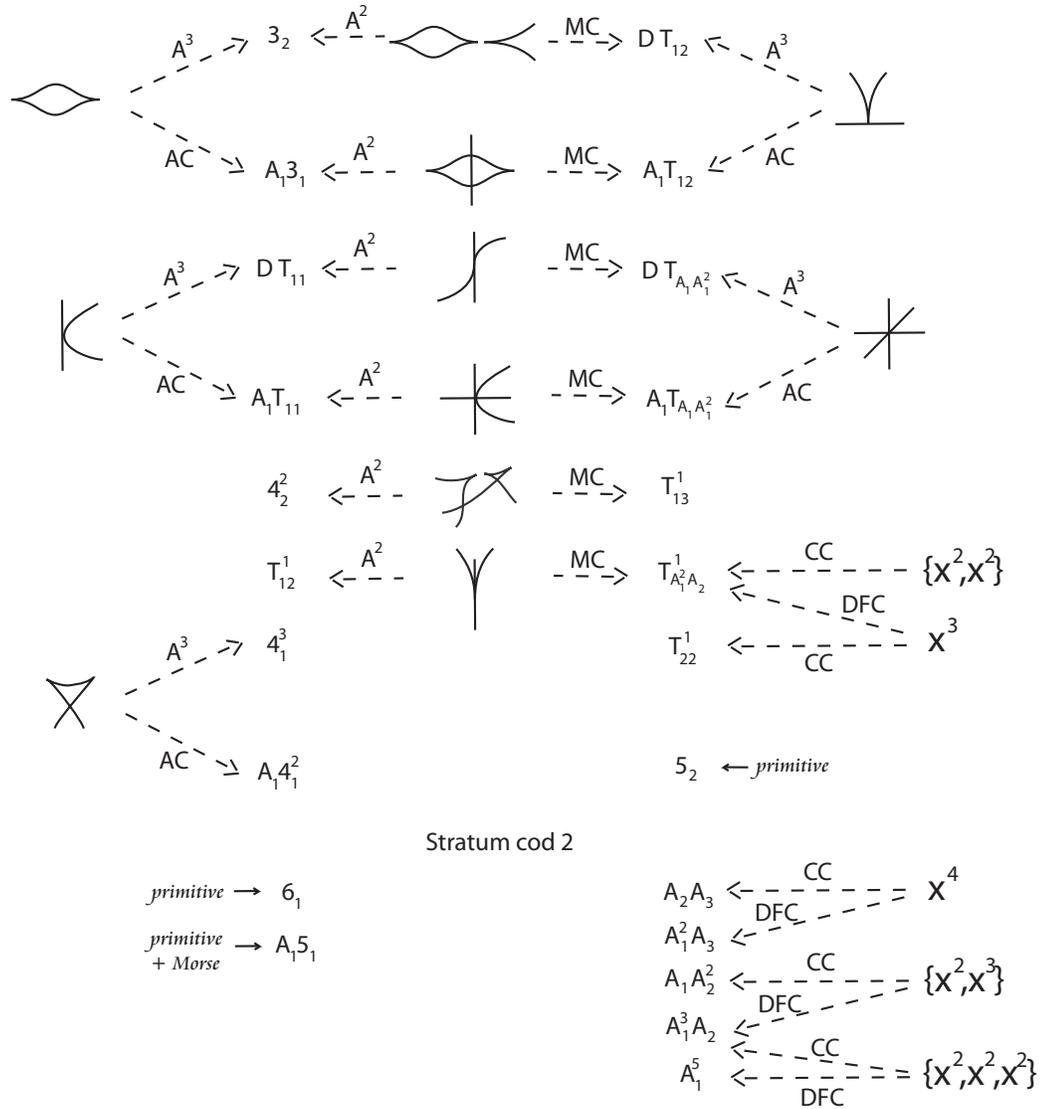}
\caption{Codimension 2 germs and multigerms of maps from $\C^3$ to
$\C^3$. $CC$ and $DFC$ stand for cuspidal concatenation and double
fold concatenation respectively.} \label{de3en3}
\end{figure}

Figure \ref{de3en3} shows how to obtain all the codimension 2
multigerms from $\C^3$ to $\C^3$ using the operations defined
starting from monogerms and the special multigerms.

 \noindent
R. Oset Sinha (raul.oset@uv.es), M. A. S. Ruas (maasruas@icmc.usp.br), R. Wik Atique (rwik@icmc.usp.br)\\
Instituto de Ci\^encias Matem\'aticas e de Computa\c{c}\~ao - USP,\\
Av. Trabalhador s\~ao-carlense, 400 - Centro,
CEP: 13566-590 - S\~ao Carlos - SP, Brazil.\\


\begin{thebibliography}{G-M-R-R}

\bibitem {Arnold3} {\sc{V. I. Arnol´d}}
{\it{Critical points of smooth functions and their normal forms}}.
Russian Math. Surveys 30 (1975) (or in Singularity Theory, LMS
Lecture Note Series 53, Cambridge UP (1981)).

\bibitem {bruce} {\sc{J. W. Bruce}} {\it {Envelopes and characteristics}}. Math. Proc. Cambridge Philos. Soc. 100 (1986), no. 3, 475–492.

\bibitem {catiana} {\sc{C. Casonatto, M. C. Romero Fuster and R. Wik Atique}} {\it{First order local invariants of immersions from 3-manifolds to $\mathbb R^4$}}. Topology and its Applications. 159 (2012), pp. 420-429.


\bibitem {robertamond} {\sc{T. Cooper, D. Mond and R. Wik Atique}} {\it{Vanishing topology of codimension 1 multi-germs over R and C}}. Compositio Math 131 (2002), no. 2, 121–160.

\bibitem {damon} {\sc{J. Damon}} {\it{$\mathcal A$-equivalence and equivalence of sections of images and discriminants}}. In: Singularity Theory and Applications (Warwick 1989), Lecture Notes in Math. 1462, Springer, New York, 1991, pp. 93–121.

\bibitem {Goryunov} {\sc{V. Goryunov}} {\it{Singularities of projections of complete
intersections}}. J. Soviet Math. 27 (1984), 2785–2811.

\bibitem {Goryunov3} {\sc{V. Goryunov}} {\it{Monodromy of the image of the mapping $\C^2\rightarrow\C^3$}}. (Russian) Funktsional. Anal. i Prilozhen. 25 (1991), no. 3, 12--18, 95; translation in Funct. Anal. Appl. 25 (1991), no. 3, 174–180 (1992).

\bibitem {Goryunov2} {\sc{V. Goryunov}} {\it{Local invariants of mappings of
surfaces into three space}}. The Arnol´d-Gelfand mathematical
seminars 223-225. Birkhauser, Boston, (1997).

\bibitem {hobbskirk} {\sc{C. A. Hobbs and N. P. Kirk}} {\it{On the classification and bifurcation of multigerms of maps from surfaces to 3-space}}. Math. Scand. 89 (2001), no. 1, 57–96.

\bibitem {houston2} {\sc{K. Houston}} {\it{On singularities of folding maps and augmentations}}.  Math. Scand. 82 (1998), no. 2, 191–206.

\bibitem {houston} {\sc{K. Houston}} {\it{Augmentation of singularities of smooth mappings}}. Internat. J. Math. 15 (2004), no. 2, 111–124.

\bibitem {houstonlittle} {\sc{K. Houston and D. Littlestone}} {\it{Vector fields liftable over corank 1 stable maps}}. Preprint (2009).

\bibitem {houstonroberta} {\sc{K. Houston and R. Wik Atique}} {\it{${\cal A}$-classification of map-germs via $_V{\cal K}$-equivalence}}. Preprint (2012).

\bibitem {kolgushkin} {\sc{P. A. Kolgushkin and R. R. Sadykov}} {\it{Simple singularities of multigerms of curves.}} Rev. Mat. Complut. 14 (2001) 311–344.

\bibitem {klotz} {\sc{C. Klotz, O. Pop and J. Rieger}}
{\it{Real double-points of deformations of $\mathcal{A}$-simple
map-germs from $\Bbb{R}\sp n$ to $\Bbb{R}\sp {2n}$}}. Math. Proc.
Cambridge Philos. Soc. 142 (2007), no. 2, 341--363.

\bibitem {mancini} {\sc{S. Mancini; M. A. S. Ruas; M. A. Teixeira}} {\it{On divergent diagrams of finite codimension}}. Port. Math. (N.S.) 59 (2002), no. 2, 179–194.

\bibitem {marartari} {\sc{W. L. Marar and F. Tari}}. {\it{On the
geometry of simple germs of co-rank $1$ maps from $\R^3$ to
$\R^3$}}. Math. Proc. Cambridge Philos. Soc. 119 (1996), no. 3,
469--481.

\bibitem {mather}{\sc{J. N. Mather}}
{\it{Stability of $\mathcal C^{\infty}$ mappings IV: Classification
of stable maps by $\R$-algebras}}. Publ. Math. IHES, 37, (1983),
223-248.

\bibitem {mond}{\sc{D. Mond}}
{\it{On the Classification of Germs of Maps From $\R^2$ to $\R^3$}}.
Proc. London Math. Soc. (3), 50, 333-369, (1983).

\bibitem {Mon} {\sc{J. A. Montaldi}}. {\it{On contact between submanifolds}}.
Michigan Math. J. 33 (1986), no. 2, 195--199.

\bibitem {morin}{\sc{B. Morin}}
{\it{Formes canoniques des singularités d'une application
différentiable}}. C. R. Acad. Sci. Paris 260 (1965) 5662–5665 and
6503–6506.

\bibitem {nishilift} {\sc{T. Nishimura}}. {\it{Vector fields liftable over finitely determined multigerms of corank at most one}}. Preprint.

\bibitem {Ohmoto}{\sc{T. Ohmoto and F. Aicardi}} {\it{First order local invariants
of apparent contours}}. Topology 45 (2006), no. 1, 27--45.

\bibitem {mio}{\sc{R. Oset Sinha and M. C. Romero Fuster}} {\it{First order local invariants of stable maps from $3$-manifolds to $\mathbb{R}^3$}}. Michigan Math. J. Volume 61, Issue 2 (2012), 385--414.

\bibitem {rieger} {\sc{J. H. Rieger}}
{\it{Families of Maps From the Plane to the Plane}}. J. London Math.
Soc. (2) 36, 351-369, (1986).

\bibitem {riegerruas} {\sc{J. H. Rieger and M. A. S. Ruas}}
{\it{Classification of $\mathcal A$-simple germs from $\K^n$ to
$\K^2$}}.  Compositio Math. 79 (1991), no. 1, 99–108.

\bibitem {wall} {\sc{C. T. C. Wall}} {\it{Finite determinacy of smooth map-germs}}. Bull. London Math. Soc. 13 (1981), 481–539.

\bibitem {roberta} {\sc{R. Wik Atique}} {\it{On the classification of multi-germs of maps from
      $\mathbb {C}^2$ to $\mathbb {C}^3$ under $\mathcal
      {A}$-equivalence}}. in J.W.Bruce and F.Tari(eds.) {\it Real and
Complex Singularities},
      Research Notes in Maths Series, Chapman \& Hall / CRC (2000),
119-133.

\bibitem {yamamoto} {\sc{M. Yamamoto}} {\it{First order semi-local invariants of stable maps
of 3-manifolds into the plane}}. Proc. London Math. Soc. (3) 92
(2006), no. 2, 471--504.

 \end{thebibliography}
\end{document}